\documentclass[12pt,a4paper]{article}
\usepackage{epsfig}
\usepackage{amsfonts}
\usepackage{amsmath}
\usepackage{amssymb}
\usepackage{graphicx}
\usepackage{amsfonts}
\usepackage{float}
\usepackage[latin1]{inputenc}
\usepackage{blkarray}
\usepackage{color}
\usepackage{amsmath,bm}
\usepackage{subfigure}
\usepackage{amssymb}
\usepackage{array}
\usepackage{multicol}
\usepackage{pdfpages}
\usepackage{subfigure}
\usepackage{multirow}
\usepackage{enumitem}
\usepackage{bbold}

\def\N{\mathbb N}
\def\R{\mathbb R}

\def\sm{\setminus}
\def\wt{\widetilde}
\def\vp{\varphi}
\def\ov{\overline}
\def\ul{\underline}

\def\({\Bigl (}
\def\){\Bigr )}
\def\div{\mbox{\rm div}}
\def\dsp{\displaystyle}
\def\x{{\bf x}}
\def\y{{\bf y}}
\def\n{{\bf n}}
\def\s{{s}}
\def\p{{\bf p}}
\def\q{{\bf q}}

\def\G{{\Gamma}}
\def\O{{\Omega}}
\def\d{{\rm d}}

\def\C{\mathcal C}
\def\D{{\cal D}}
\def\cells{{\cal M}}
\def\faces{{\cal F}}
\def\nodes{{\cal V}}
\def\edges{{\cal E}}
\def\T{{\cal T}}

\def\to{\rightarrow} 
 
\def\trace{\gamma} 
\def\grad{{\nabla}}

\def\OG{\O\sm\ov\G}

\def\dofm{\mathit{dof}_{\D_m}}
\def\doff{\mathit{dof}_{\D_f}}
\def\dofDirm{\mathit{dof}_{Dir_m}}
\def\dofDirf{\mathit{dof}_{Dir_f}}
\def\dofDir{\mathit{dof}_{Dir}}
\def\del{\partial}

\renewcommand{\restriction}{\mathord{\upharpoonright}}
\newcommand{\ssubset}{\subset\joinrel\subset}

\newtheorem{remark}{Remark}[section]
\newtheorem{definition}{Definition}[section]
\newtheorem{lemma}{Lemma}[section]
\newtheorem{proposition}{Proposition}[section]

\newtheorem{corollary}{Corollary}[section]

\newenvironment{proof}[1][Proof]{\begin{trivlist}
\item[\hskip \labelsep {\bfseries #1}]}{\end{trivlist}}

\newcommand{\qed}{\hfill \ensuremath{\Box}}

\begin{document}

\title{Gradient discretization of Hybrid Dimensional Darcy Flows in Fractured Porous Media 
with discontinuous pressures at the matrix fracture interfaces}

\author{
K. Brenner \thanks{
Laboratoire de Math\'ematiques J.A. Dieudonn\'e, 
UMR 7351 CNRS, University Nice Sophia Antipolis, and team COFFEE, INRIA Sophia Antipolis
M\'editerran\'ee, Parc Valrose 06108 Nice Cedex 02, France, \{konstantin.brenner, julian.hennicker, roland.masson\}@unice.fr } , 
J. Hennicker {\footnotemark[1] }\thanks{CSTJF, TOTAL S.A. - 
Avenue Larribau, 64018 Pau, France} , 
R. Masson {\footnotemark[1] },
P. Samier {\footnotemark[2] }
}

\maketitle

\begin{abstract}
{We investigate the discretization of Darcy flow through fractured porous media on general meshes. 
We consider a hybrid dimensional model, invoking a complex network of planar fractures. The model accounts for matrix-fracture interactions and fractures acting either as drains or as barriers, i.e. we have to deal with pressure discontinuities at matrix-fracture interfaces. 
The numerical analysis is performed in the general framework of gradient discretizations which is 
extended to the model under consideration. 
Two families of schemes namely the Vertex Approximate Gradient scheme (VAG) and the Hybrid 
Finite Volume scheme (HFV) are detailed and shown to satisfy the gradient scheme framework, 
which yields, in particular, convergence.
Numerical tests confirm the theoretical results.}
{Gradient Discretization; Darcy Flow, Discrete Fracture Networks, Finite Volume}
\end{abstract}

\section{Introduction}

This work deals with the discretization of 
Darcy flows in fractured porous media for which the fractures are 
modelized as interfaces of codimension one. In this framework, the $d-1$ dimensional 
flow in the fractures is 
coupled with the $d$ dimensional flow in the matrix leading to the so called,  
hybrid dimensional Darcy flow model. 
We consider the case for which the pressure 
can be discontinuous at the matrix fracture interfaces in order to account for fractures 
acting either as drains or as barriers as described in \cite{FNFM03}, \cite{MJE05} and \cite{ABH09}. 
In this paper, we will study the family of models described in \cite{MJE05} and \cite{ABH09}.

It is also assumed in the following that the pressure is continuous 
at the fracture intersections. This corresponds  
to a ratio between the permeability at the fracture intersection and the width 
of the fracture assumed to be high compared with the ratio between the tangential 
permeability of each fracture and its length. 
We refer to \cite{FFSR14} for a more general reduced model 
taking into account discontinuous pressures at fracture intersections in dimension $d=2$. \\

The discretization of such hybrid dimensional Darcy flow model 
has been the object of several works. 
In \cite{FNFM03}, \cite{KDA04}, \cite{ABH09} a cell-centered Finite Volume scheme using a 
Two Point Flux Approximation (TPFA) is proposed assuming the orthogonality of the mesh and 
isotropic permeability fields. Cell-centered Finite Volume schemes 
have been extended to general meshes and anisotropic permeability fields 
using  MultiPoint Flux Approximations (MPFA)  in \cite{TFGCH12}, \cite{SBN12}, and \cite{AELH14}. 
In \cite{MJE05}, a Mixed Finite Element (MFE) method is proposed and a 
MFE discretization adapted to non-matching fracture and matrix meshes is studied 
in \cite{DS12}. More recently the Hybrid Finite Volume (HFV) scheme, introduced in \cite{EGH09},  
has been extended in \cite{FFJR15} for the non matching 
discretization of two reduced fault models.  Also a Mimetic Finite Difference (MFD) scheme 
is used in \cite{AFSVV15} in the matrix domain coupled with a TPFA scheme in the fracture network.  
Discretizations of the related reduced model \cite{MAE02} assuming a 
continuous pressure at the matrix fracture interfaces have been proposed 
in \cite{MAE02} using a MFE method, in \cite{RJBH06} using a 
Control Volume Finite Element method (CVFE), in \cite{GSDFN} using the HFV scheme, 
and in \cite{GSDFN, BGGM14} using an extension of the 
Vertex Approximate Gradient (VAG) scheme introduced in \cite{Eymard.Herbin.ea:2010}.

In terms of convergence analysis, 
the case of continuous pressure models 
at the matrix fracture interfaces \cite{MAE02} is studied in 
\cite{GSDFN} for a general fracture network but 
the current state of the art for the discontinuous pressure models at the 
matrix fracture interfaces is still limited to rather simple geometries. 
Let us recall that the family of models introduced in \cite{MJE05} and \cite{ABH09}
depends on a quadrature parameter denoted by $\xi\in [\frac{1}{2},1]$ 
for the approximate integration in the width of the fractures.
Existing convergence analysis for such models cover
the case of one non immersed fracture separating the domain into two subdomains using 
a MFE discretization in \cite{MJE05} or a non matching MFE discretization in \cite{DS12} 
for the range $\xi\in (\frac{1}{2},1]$. 
In \cite{ABH09}, the case of one fully immersed fracture in dimension $d=2$ using a TPFA discretization 
is analysed for the full range of parameters $\xi\in [\frac{1}{2},1]$. \\

The main goal of this paper is to study the discretizations of such models and their convergence properties 
by extension of the gradient scheme framework. 
The gradient scheme framework has been introduced in \cite{Eymard.Herbin.ea:2010}, 
\cite{DEGH}, \cite{koala} to analyse  the convergence of numerical 
methods for linear and nonlinear second order diffusion problems. As shown in \cite{DEGH},  this 
framework accounts for various conforming and non conforming 
discretizations such as Finite Element methods, Mixed and Mixed Hybrid Finite Element methods, and 
some Finite Volume schemes like symmetric MPFA, the VAG schemes \cite{Eymard.Herbin.ea:2010}, 
and the HFV schemes \cite{EGH09}. 

Our extension of the gradient scheme framework to the hybrid dimensional Darcy flow 
model will account for general fracture networks including fully, partially and non immersed fractures as 
well as fracture intersections in a 3D surrounding matrix domain. Each individual fracture will 
be assumed to be planar. 
The framework will cover the range of parameters $\xi\in (\frac{1}{2},1]$ 
excluding the value $\xi=\frac{1}{2}$ in order to allow for a primal variational formulation. 

Two examples of gradient discretizations will be provided, namely the extension of 
the VAG and HFV schemes defined in 
\cite{Eymard.Herbin.ea:2010} and \cite{EGH09} to the family of hybrid dimensional Darcy flow models. 
In both cases, it is assumed that the fracture network is conforming to the mesh in the sense that it is 
defined as a collection of faces of the mesh. The mesh is assumed to be polyhedral with possibly non 
planar faces for the VAG scheme and planar faces for the HFV scheme. 
Two versions of the VAG scheme will be studied, the first corresponding to 
the conforming ${\mathbb P}_1$ finite element on a tetrahedral submesh, and the second 
to a finite volume scheme using lumping for the source terms as well as for 
the matrix fracture fluxes. 
The VAG scheme has the advantage to lead 
to a sparse discretization on tetrahedral or mainly tetrahedral meshes. 
It will be compared to the HFV discretization 
using face and fracture edge unknowns in addition to the cell unknowns. 
Note that the HFV scheme of \cite{EGH09} has been generalized 
in \cite{DEGH10} as the family of Hybrid Mimetic Mixed methods which which encompasses 
the family of MFD schemes \cite{Bre05}. 
In this article, we will focus without restriction on 
the particular case presented in \cite{EGH09} for the sake of simplicity. \\

In section \ref{sec_model} we introduce the geometry of the matrix and fracture domains 
and present the strong and weak formulation of the model. 
Section \ref{sec_GS} is devoted to the introduction of the general framework 
of gradient discretizations and the derivation of the error estimate \ref{properror}.
In section \ref{sec_VAGHFV} we define and investigate the families of VAG and HFV discretizations.
Having in mind applications to multi-phase flow, 
we also present a Finite Volume formulation involving conservative fluxes, 
which applies for both schemes. In section \ref{sec_num}, the VAG and HFV schemes are compared 
in terms of accuracy and CPU efficiency for both Cartesian and tetrahedral meshes on hererogeneous 
isotropic and anisotropic media using a family of analytical solutions.

\section{Hybrid dimensional Darcy Flow Model in Fractured Porous Media}
\label{sec_model}
\subsection{Geometry and Function Spaces}

Let $\Omega$ denote a bounded domain of $\R^d$, $d=2,3$ 
assumed to be polyhedral for $d=3$ and polygonal for $d=2$. 
To fix ideas the dimension will be fixed to $d=3$ when it needs to be specified, 
for instance in the naming of the geometrical objects or for the space discretization 
in the next section. The adaptations to the case $d=2$ are straightforward. \\

Let 
$
\overline \Gamma = \bigcup_{i\in I} \overline \Gamma_i
$  
and its interior $\Gamma = \overline \Gamma\setminus \partial\overline\Gamma$ 
denote the network of fractures $\Gamma_i\subset \Omega$, $i\in I$, such that each $\Gamma_i$ is 
a planar polygonal simply connected open domain included in a plane ${\cal P}_i$ of $\R^d$. 
It is assumed that the angles of $\Gamma_i$ 
are strictly smaller than $2\pi$, and that $\Gamma_i\cap\overline\Gamma_j=\emptyset$ for all $i\neq j$ .

For all $i\in I$, 
let us set $\Sigma_i = \partial\Gamma_i$, with $\n_{\Sigma_i}$ as unit vector in $\mathcal P_i$, normal to $\Sigma_i$ and outward to $\G_i$. Further $\Sigma_{i,j}= \Sigma_i\cap\Sigma_j$, $j\in I\sm\{i\}$, 
$\Sigma_{i,0} = \Sigma_i\cap\partial\Omega$, 
$\Sigma_{i,N} = \Sigma_i\setminus(\bigcup_{j\in I\sm\{i\}}\Sigma_{i,j}\cup \Sigma_{i,0})$, 
$\Sigma = \bigcup_{(i,j)\in I\times I, i\neq j} ( \Sigma_{i,j}\setminus\Sigma_{i,0} )$ 
and $\Sigma_0 = \bigcup_{i\in I} \Sigma_{i,0}$. 
It is assumed that $\Sigma_{i,0} = \overline\Gamma_i\cap\partial\Omega$. 

\begin{figure}[h!]
\begin{center}
\includegraphics[height=0.25\textwidth]{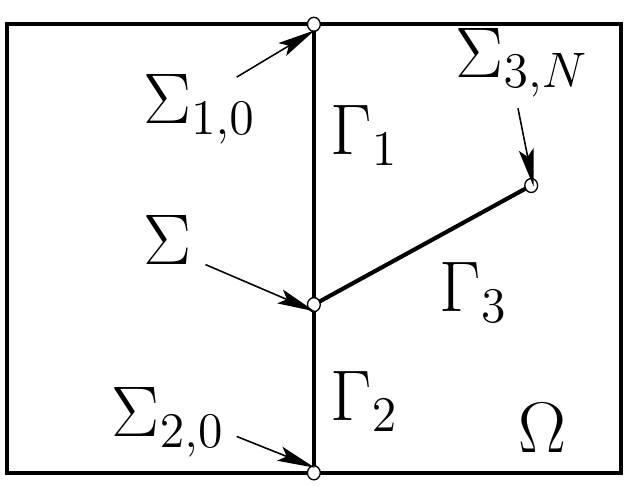}\hspace*{1cm}
\includegraphics[height=0.25\textwidth]{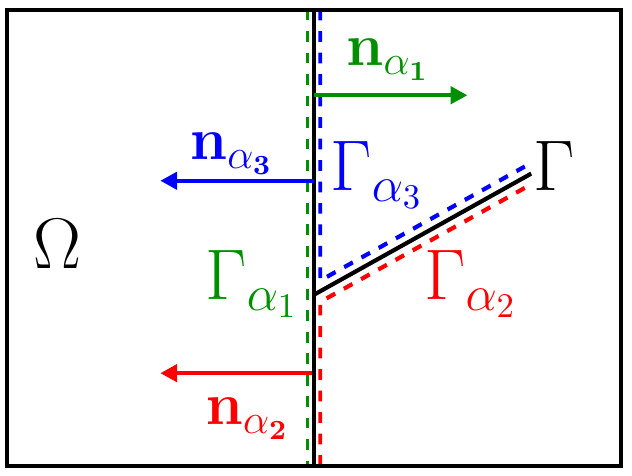}   
\caption{Example of a 2D domain $\Omega$ and 3 intersecting fractures $\Gamma_i,i=1,2,3$. 
We might define the fracture plane orientations by $\alpha^+(1) = \alpha_1, \alpha^-(1) = \alpha_3$ for $\Gamma_1$, 
$\alpha^+(2) = \alpha_1, \alpha^-(2) = \alpha_2$ for $\Gamma_2$, and $\alpha^+(3) = \alpha_3, \alpha^-(3) = \alpha_2$ for 
$\Gamma_3$.}
\label{fig_network}
\end{center}
\end{figure}

We will denote by $\d\tau({\bf x})$ the $d-1$ dimensional Lebesgue measure on $\Gamma$.
On the fracture network $\G$, we define the function space 
$
L^2(\Gamma) = \{v = (v_i)_{i\in I}, v_i\in L^2(\Gamma_i), i\in I\}, 
$
endowed with the norm $\|v\|_{L^2(\Gamma)} = (\sum_{i\in I} \|v_i\|^2_{L^2(\Gamma_i)})^{1\over 2}$ and its subspace $H^1(\Gamma)$ consisting of functions $v = (v_i)_{i\in I}$ such that $v_i\in H^1(\Gamma_i)$, $i\in I$ with continuous traces at the fracture intersections $\Sigma_{i,j}$, $j\in I\sm\{i\}$. The space $H^1(\Gamma)$ is 
endowed with the norm $\|v\|_{H^1(\Gamma)} = (\sum_{i\in I} \|v_i\|^2_{H^1(\Gamma_i)})^{1\over 2}$. We also define it's subspace with vanishing traces on $\Sigma_0$, which we denote by $H^1_{\Sigma_0}(\Gamma)$. 

On $\Omega\backslash\ov\G$, the gradient operator 
from $H^1(\Omega\backslash\ov\G)$ to $L^2(\O)^d$ is denoted by $\nabla$. On the fracture 
network $\G$, the tangential gradient, acting from 
$H^1(\Gamma)$ to $L^2(\Gamma)^{d-1}$, is denoted by $\nabla_\tau$, and such that 
$$
\nabla_\tau v = (\nabla_{\tau_i} v_i)_{i\in I},  
$$
where, for each $i\in I$, the tangential gradient $\nabla_{\tau_i}$ 
is defined from $H^1(\G_i)$ to $L^2(\G_i)^{d-1}$ by fixing a 
reference Cartesian coordinate system of the plane ${\cal P}_i$ containing $\G_i$. 
We also denote by $\div_{\tau_i}$ the divergence operator from 
$H_{\div}(\Gamma_i)$ to $L^2(\Gamma_i)$. \\

We assume that there exists a finite family $(\G_\alpha)_{\alpha\in\chi}$ such that for all $\alpha\in\chi$ holds:
$\G_\alpha\subset\G$ and there exists a lipschitz domain $\omega_\alpha\subset\OG$, such that $\G_\alpha = \del\omega_\alpha\cap\G$. 
For $\alpha\in\chi$ and an apropriate choice of $I_\alpha\subset I$ we assume that $\ov\G_\alpha = \bigcup_{i\in I_\alpha}\ov\G_i$. 
Furthermore should hold $\ov\G = \bigcup_{\alpha\in\chi}\ov\G_\alpha$. 
We also assume that each $\G_i\subset\G$ is contained in $\G_\alpha$ for exactly two $\alpha\in\chi$ and that we can define a unique mapping $i\longmapsto (\alpha^+(i),\alpha^-(i))$ from $I$ to $\chi\times\chi$, such that $\G_i\subset\G_{\alpha^+(i)}\cap\G_{\alpha^-(i)}$ and $\alpha^+(i)\neq\alpha^-(i)$ (cf. figure \ref{fig_network}). 
For all $i\in I$, $\alpha^\pm(i)$ defines the two sides of the fracture $\Gamma_i$ in $\Omega\setminus\overline\Gamma$ 
and we can introduce the corresponding unit normal vectors $n_{\alpha^\pm(i)}$ at $\Gamma_i$ 
outward to $\omega_{\alpha^\pm(i)}$, 
such that $\n_{\alpha^+(i)} + \n_{\alpha^-(i)} = 0$. We therefore obtain for $\alpha\in\chi$ and a.e. $\x\in\G_\alpha$ a unique 
unit normal vector $\bf n_\alpha(\x)$ outward to $\omega_\alpha$. 
A simple choice of $(\G_\alpha)_{\alpha\in\chi}$ is given by both sides of each fracture $i\in I$ but 
more general choices are also possible such as for example the one exhibited in figure \ref{fig_network}.\\

Then, for $\alpha\in\chi$, we can define the trace operator on $\G_\alpha$: 
$$
\gamma_{\alpha} : H^1(\Omega\setminus\overline\G) \rightarrow L^2(\G_\alpha),  
$$
and the normal trace operator on $\G_\alpha$ outward to the side $\alpha$: 
$$
\gamma_{\n,\alpha} : H_{\div} (\Omega\setminus\overline\G) \rightarrow \D'(\G_\alpha). 
$$

We now define the hybrid dimensional function spaces that 
will be used as variational spaces for the Darcy flow model 
in the next subsection: 
$$
V = H^1(\Omega\setminus \overline\G)\times H^1(\G),  
$$
and its subspace 
$$
V^0 = H^1_{\partial\Omega}(\O\setminus \overline\G)\times H^1_{\Sigma_0}(\G),
$$
where (with $\trace_{\del\O}\colon H^1(\O\backslash\ov\G)\rightarrow L^2(\del\O)$ denoting the trace operator on $\del\O$)
$$
H^1_{\partial\Omega}(\O\setminus \overline\G) = \{ v\in H^1(\O\backslash\ov\G)\mid \trace_{\del\O} v = 0\text{ on }\del\Omega\},
$$
as well as
$$
W = W_m\times W_f,
$$
where
\begin{align*}
W_m &= \left\{ \q_m \in H_{\div}(\Omega\setminus\overline\Gamma)\mid\gamma_{\n,\alpha}\q_m \in L^2(\G_\alpha) 
\text{ for all }\alpha\in\chi\right\}
\text{ and }\\
 W_f &= \{\q_f=(\q_{f,i})_{i\in I} \mid \q_{f,i} \in H_{\div}(\G_i) \mbox{ for all } i\in I\\
 &\text{ and }\sum_{i\in\G}\int_{\G_i}\(\grad_\tau v\cdot\q_{f,i} + v\cdot\div_{\tau_i}\q_{f,i}\) \d\tau(\x) = 0 \mbox{ for all } v\in H_{\Sigma_0}^1(\G) \}.
\end{align*}

On $V$, we define the positive semidefinite, symmetric bilinear form
\begin{align*}
( (u_m,u_f),(v_m,v_f))_V &= 
\dsp\int_{\O}\grad u_m \cdot\grad v_m \d\x
+\dsp\int_{\G}\grad_\tau u_f \cdot\grad_\tau v_f \d\tau(\x)\\
& + \dsp \sum_{\alpha\in\chi} \int_{\G_\alpha} (\gamma_{\alpha} u_m - u_f)(\gamma_{\alpha} v_m - v_f)d\tau(\x)
\end{align*}
for $(u_m,u_f), (v_m,v_f)\in V$, which induces the seminorm $|(v_m,v_f)|_{V}$.
Note that $(\cdot,\cdot)_V$ is a scalar product and $|\cdot|_V$ is a norm on $V^0$, denoted by $\|\cdot\|_{V^0}$ in the following.\\

We define for all $(\p_m,\p_f), (\q_m,\q_f)\in W$ the scalar product 
\begin{align*}
((\p_m, \p_f),(\q_m, \q_f))_W &= \dsp\int_{\O}\p_m \q_m \d\x + \dsp\int_{\O}\div \p_m \cdot\div \q_m \d\x\\
 &+ \dsp\int_{\G} \p_f \q_f \d\tau(\x) + \dsp\int_{\G}\div_\tau \p_f \cdot\div_\tau \q_f \d\tau(\x)\\
& + \dsp \sum_{\alpha\in\chi} \int_{\G_\alpha} (\gamma_{\n,\alpha} \p_m\cdot\gamma_{\n,\alpha} \q_m)d\tau(\x),
\end{align*}
which induces the norm $\|(\q_m,\q_f)\|_{W}$, and where we have used the notation $\div_{\tau} \p_f = \div_{\tau_i} \p_{f,i}$ on $\Gamma_i$ for all $i\in I$ and $\p_f = (\p_{f,i})_{i\in I} \in W_f$. \\

Using similar arguments as in the proof of \cite{Raviart}, example II.3.4, one can prove the following Poincar\'e type inequality.

\begin{proposition}\label{proppoincarecont}
The norm $\|\cdot\|_{V^0}$ satisfies the following inequality
\begin{equation}
\label{poincarecont}
\|v_m\|_{H^1(\Omega\setminus \overline\G)} + \|v_f\|_{H^1(\G)} \leq \mathcal C_P  \|(v_m,v_f)\|_{V^0},
\end{equation}
for all $(v_m,v_f)\in V^0$. 
\end{proposition}

\begin{proof}
We apply the ideas of the proof of \cite{Raviart}, example II.3.4 and assume that the statement of the proposition is not true. 
Then we can define a sequence $(v_l)_{l\in\N}$ in $V^0$, such that
\begin{equation}
\label{proofproppoincarecont1}
\|v_l\|_{H^1} = 1 \qquad\text{and}\qquad \|v_l\|_{V^0} < \frac{1}{l},
\end{equation}
where, for this proof, $\|\cdot\|_{H^1} = \|\cdot\|_{H^1(\Omega\setminus \overline\G)} + \|\cdot\|_{H^1(\G)}$. The imbedding
$$
(V^0,\|\cdot\|_{H^1}) \hookrightarrow \(L^2(\O)\times L^2(\G),
\|\cdot\|_{L^2(\O)} + \|\cdot\|_{L^2(\G)}\)
$$
is compact, provided that $\Omega\backslash\ov\G$ has the cone property (see \cite{Adams}, theorem 6.2). Thus, there is a subsequence $(v_\mu)_\mu$ of 
$(v_l)_{l\in\N}$ and  $v\in L^2(\O)\times L^2(\G)$, such that
$$
v_\mu\longrightarrow v \qquad\text{in }L^2(\O)\times L^2(\G).
$$
On the other hand it follows from (\ref{proofproppoincarecont1}) that
\begin{align*}
\grad v_{m_\mu}&\longrightarrow 0\qquad\text{in }L^2(\O)\\
\grad_\tau v_{f_\mu}&\longrightarrow 0\qquad\text{in }L^2(\G).
\end{align*}
Since $(V^0, \|\cdot\|_{H^1})$ is complete, we have
$$
v_\mu\longrightarrow v \qquad\text{in }V^0,
$$
with
$$
\|v\|_{V^0} = \lim_{\mu\to\infty}\|v_\mu\|_{V^0} = 0.
$$
Since $\|\cdot\|_{V^0}$ is a norm on $V^0$, we have $v = 0 \in V^0$, but $\|v\| = 1$, which is a contradiction.
\hfill\qed
\end{proof}

\begin{remark}
With the precedent proof it is readily seen that inequality
(\ref{poincarecont}) holds for all functions $v\in V$ whose trace vanishes on a subset of 
$\del(\Omega\backslash\ov\G)$ with positive surface measure. The requirement is that $v$ has to be in 
a closed subspace of $(V,\|\cdot\|_{H^1})$ for which $\|\cdot\|_{V^0}$ is a well defined norm.	
\end{remark}

The convergence analysis presented in section \ref{sec_VAGHFV} requires some results on the 
density of smooth subspaces of $V$ and $W$, which we state below.

\begin{definition}
\begin{enumerate}
\item $C^\infty_{\Omega}$ is defined as the subspace of functions in $C_b^\infty(\OG)$ vanishing on a neighbourhood 
of the boundary $\partial\Omega$, where $C_b^\infty(\OG)\subset C^\infty(\OG)$ is the 
set of functions $\varphi$, such that for all $\x\in\O$ there exists $r > 0$, such that for all 
connected components $\omega$ of $\{\x + \y\in\R^d\mid |\y| < r\}\cap(\OG)$ one has $\varphi\in C^\infty(\ov\omega)$.
\item $C^\infty_{\G} = \gamma_\Gamma (C_0^\infty(\O))$ is defined as the image of $C_0^\infty(\O)$ of the trace operator $\gamma_\Gamma\colon H_0^1(\O)\to L^2(\G)$.
\item $C^\infty_{W_m} = {C_b^\infty(\OG)}^d$.
\item $C_{W_f}^\infty = \{ \q_f = (\q_{f,i})_{i\in I}\mid \q_{f,i}\in {C^\infty(\ov\G_i)}^{d-1},\ \sum_{i\in I}\q_{f,i}\cdot\n_{\Sigma_i} = 0\text{ on }\Sigma,\ \q_{f,i}\cdot\n_{\Sigma_i} = 0\text{ on }\Sigma_{i,N},\ i\in I\}$.
\end{enumerate}
\end{definition}
Let us first state the following Lemma that will be used to prove the density of $C^\infty_{W_m}\times C^\infty_{W_f}$ 
in $W$.
\begin{lemma}
\label{lemmaweakderivatives}
Let $v_m\in L^2(\O),\ v_f\in  L^2(\G),\ G\in L^2(\O)^d,\ H\in L^2(\G)^{d-1}$ and 
$J_\alpha\in L^2(\G_\alpha),\ \alpha\in\chi$ such that 
\begin{align}
\label{eqlemmaweakderivatives}
\int_\O(G\cdot\q_m + v_m\div\q_m) \d\x + \int_\G (H\cdot\q_f + v_f\div_\tau\q_f) \d\tau(\x)
+ \sum_{\alpha\in\chi}\int_{\G_\alpha}\trace_{\n,\alpha}\q_m \d\tau(\x)
(J_\alpha - v_f) = 0
\end{align}
for all $(\q_m, \q_f)\in C^\infty_{W_m}\times C^\infty_{W_f}$. Then holds
$(v_m,v_f)\in V^0$, $(G,H) = (\grad v_m,\grad_\tau v_f)$ and $J_\alpha = v_f - \trace_\alpha v_m  \text{ for } \alpha\in\chi$.
\end{lemma}

\begin{proof}
Firstly, for all $\q_m\in C_0^\infty(\O\backslash\ov\G)^d$, we have
$$
\int_\O(G\cdot\q_m + v_m\div\q_m) \d\x = 0
$$
and therefore $v_m\in H^1(\O\backslash\ov\G)$ and $\grad v_m = G$.
\par For a.e. $\x\in\del\O$, there exists an open planar domain $\omega\ssubset\del\O\backslash\del\G$ containing $\x$ such that for all 
$f\in C_0^{\infty}(\omega)$ there exists $\q_m\in C_{W_m}^\infty$ with 
\begin{align*}
\trace_{\n_{\del\O}}\q_m &= \left\{
\begin{array}{l l}
f &\quad \text{on }\omega,\\
0 &\quad \text{on }\del\O\backslash\omega,
\end{array}\right.\\
\trace_{\n,\alpha} \q_m &= 0\quad\text{on }\G_\alpha,\ \alpha\in\chi,
\end{align*}
where $\trace_{\n_{\del\O}}$ denotes the normal trace operator on the boundary of $\O$. From \eqref{eqlemmaweakderivatives}, taking $\q_f = 0$, we obtain
$$
0 = \int_\O(\grad v_m\cdot\q_m + v_m\div\q_m) \d\x = \int_{\del\O}\trace_{\del\O}v_m\trace_{\n_{\del\O}}\q_m \d\tau(\x) = \int_{\omega}\trace_{\del\O}v_m f \d\tau(\x).
$$
where $\trace_{\del\O}$ denotes the trace operator on the boundary of $\O$.
We deduce $\trace_{\del\O}v_m = 0$ a.e. on $\del\O\backslash\del\G$. 
Hence $v_m\in H_{\del\O}^1(\O\backslash\ov\G)$.
\par Further, for a.e. $\x\in\G_\alpha$ there exists an open planar domain $\omega_\alpha\ssubset\G_\alpha$ containing $\x$ such that for 
all $g\in C_0^{\infty}(\omega_\alpha)$ there exists $\q_m\in C_{W_m}^\infty$ with 
\begin{align*}
\trace_{\n,\alpha}\q_m &= \left\{
\begin{array}{l l}
g &\quad \text{on }\omega_\alpha,\\
0 &\quad \text{on }\G_\alpha\backslash\omega_\alpha,
\end{array}\right.\\
\trace_{\n,\beta} \q_m &= 0\quad\text{on }\G_\beta,\text{ for }\beta\neq\alpha,\\
\trace_{\n_{\del\O}} \q_m &= 0\quad\text{on }\del\O.
\end{align*}
From \eqref{eqlemmaweakderivatives} we obtain
\begin{align*}
0 &= \int_\O(\grad v_m\cdot\q_m + v_m\div\q_m) \d\x + \sum_{\alpha\in\chi}\int_{\G_\alpha}\trace_{\n,\alpha}\q_m (J_\alpha - v_f) \d\tau(\x)  \\
 &= \int_{\G_\alpha}\trace_{\n,\alpha}\q_m (J_\alpha - v_f + \trace_\alpha v_m) \d\tau(\x) 
 = \int_{\omega_\alpha} g (J_\alpha - v_f + \trace_\alpha v_m) \d\tau(\x) .
\end{align*}
We deduce $J_\alpha = v_f - \trace_\alpha v_m$ a.e. on $\G_\alpha,\ \alpha\in\chi$. 
\par Next, for all $\q_f\in C_0^\infty(\G_i)^{d-1},\ i\in I$, we have from \eqref{eqlemmaweakderivatives}
$$
\int_{\G_i}(H\cdot\q_f + v_f\div\q_f) \d\tau(\x) = 0
$$
and therefore $v_f\restriction_{\G_i}\in H^1(\G_i)$ for $i\in I$ and 
$\grad_{\tau_i} v_f\restriction_{\G_i} = H\restriction_{\G_i}$.
\par Let $i,j\in I$, $i\neq j$. For a.e. $\x\in\Sigma_{i,j}\sm\Sigma_{i,0}$ there exists an open interval 
$c_{ij}\ssubset\Sigma_{i,j}\sm\Sigma_{i,0}$ containing $\x$ such that 
for all $h\in C_0^{\infty}(c_{ij})$ there exists $\s\in C_{W_f}^\infty$ with 
\begin{align*}
\trace_{\n_{\Sigma_i}}\s &= h = -\trace_{\n_{\Sigma_j}}\s\quad\text{on }c_{ij},\\
\trace_{\n_{\Sigma_{k}}}\s &= 0\quad\text{on }\Sigma_{k}\backslash c_{ij},\ k\in I.
\end{align*}
From \eqref{eqlemmaweakderivatives} we obtain
$$
0 = \int_\G(\grad_\tau v_f\cdot\s + v_f\div_\tau\s) \d\tau(\x) = \int_{c_{ij}}(\trace_{\Sigma_i}v_f - \trace_{\Sigma_j}v_f)\trace_{\n_{\Sigma_i}}\s \d\sigma(\x),
$$
$ \d\sigma(\x)$ denoting the $d-2$ dimensional Lebesgue measure on $\Sigma$. We deduce $\trace_{\Sigma_i}v_f = \trace_{\Sigma_j}v_f$ a.e. on $\Sigma_{i,j}\sm\Sigma_{i,0},\ i,j\in I, i\neq j$.
The proof of $\trace_{\Sigma_0}v_f = 0$ a.e. on $\Sigma_0$ goes analogously.
Hence $v_f\in H_{\Sigma_0}^1(\G)$.
\hfill\qed
\end{proof}

\begin{proposition}
$C^\infty_{\Omega}\times C^\infty_{\G}$ is dense in $V^0$.
\end{proposition}
\begin{proof}
Firstly, note that we have
\begin{align*}
{1\over \sqrt{2}}\( \|\grad u_m\|_{{L^2(\O)}^d} + \|\grad_\tau u_f\|_{{L^2(\G)}^{d-1}} \)&\leq \|(u_m,u_f)\|_{V^0}\\
&\leq C(\Omega,\G)\cdot \( \|\grad u_m\|_{{L^2(\O)}^d} + \|\grad_\tau u_f\|_{{L^2(\G)}^{d-1}} \),
\end{align*}
i.e. $\|\cdot\|_{V^0}$ is equivalent to the standard norm 
$\|\grad \cdot\|_{{L^2(\O)}^d} + \|\grad_\tau \cdot\|_{{L^2(\G)}^{d-1}}$
on $V^0$. The density of $C_\O^\infty$ in $H_{\del\O}^1(\O\sm\ov\G)$ being a classical result, we are concerned to prove the density of $C_\G^\infty$ in $H_{\Sigma_0}^1(\G)$ in the following.
Since $H_{\Sigma_0}^1(\G)\subset \gamma_\Gamma(H_0^1(\O))$, we can define 
$\tilde V^0 = \gamma_\G^{-1}(H_{\Sigma_0}^1(\G))\subset H_0^1(\O)$. 
In Proposition 2 of \cite{GSDFN} it is shown that $C_0^\infty(\O)$ is dense in $( \tilde{V}^0, \|\grad\cdot\|_{L^2(\O)^d} + \|\grad_\tau\gamma_\G\cdot\|_{L^2(\G)^{d-1}} )$. Hence $C_\G^\infty$ is dense in $( H_{\Sigma_0}^1(\G), \|\grad_\tau\cdot\|_{L^2(\G)^{d-1}} )$.
\hfill\qed
\end{proof}

\begin{proposition}
$C^\infty_{W_m}\times C^\infty_{W_f}$ is dense in $W$.
\end{proposition}
\begin{proof}
Since $W_f$ is a closed subspace of the Hilbert space $\prod_{i\in I} H_{\div}(\G_i)$, any linear form $l\in W_f'$ is the 
restriction to $W_f$ of a linear form still denoted by $l$ in $\prod_{i\in I} H_{\div}(\G_i)'$. Then, for some $f\in L^2(\G)$ and $\bm g\in {L^2(\G)}^{d-1}$ holds
$$
<l,\q_f> = \sum_{i\in I}\int_{\G_i}\(\bm g\cdot\q_f + f\cdot\div_\tau\q_f\) \d\tau(\x),
$$
for all $\q_f\in W_f$. Let us assume now that $<l,\bm\vp> = 0$ for all $\bm \vp\in C_{W_f}^\infty$. 
Corresponding to Lemma \ref{lemmaweakderivatives} holds $f\in H_{\Sigma_0}^1(\G)$.
From the definition of $W_f$ we conclude that $<l,\bm\q_f> = 0$ for all $\q_f\in W_f$. 
\par Let now $l\in W_m'$. Then there exist $f\in L^2(\O),\ \bm g\in {L^2(\O)}^d$ and $h_\alpha\in L^2(\G_\alpha)\ (\alpha\in\chi)$, such that
$$
<l,\q_m> = \int_{\O}\(\bm g \cdot\q_m + f\cdot\div\q_m\) \d\x + \sum_{\alpha\in\chi}\int_{\G_\alpha} h_\alpha\trace_{\n,\alpha}\q_m \d\tau(\x),
$$
for all $\q_m\in W_m$. Furthermore, let us assume that $<l,\bm\vp> = 0$ for all $\bm \vp\in C_{W_m}^\infty$. From Lemma \ref{lemmaweakderivatives} we deduce that $f\in H_{\del\O}^1(\Omega\setminus\overline\Gamma)$, that $\bm g = \grad f$ and that $h_\alpha = \trace_\alpha f\ (\alpha\in\chi)$. Using this, we conclude, again by the rule of partial integration, that $<l,\bm\q_m> = 0$ for all $\q_m\in W_m$.
\hfill\qed
\end{proof}

\subsection{Single Phase Darcy Flow Model}

\subsubsection{Strong formulation}

In the matrix domain $\Omega\setminus\overline\G$, let us denote 
by $\Lambda_m\in L^{\infty}(\Omega)^{d\times d}$ 
the permeability tensor such that 
there exist  $\overline\lambda_m\geq \underline\lambda_m > 0$ 
with 
$$
\underline\lambda_m|\zeta|^2 \leq (\Lambda_m(\x)\zeta,\zeta) \leq \overline\lambda_m|\zeta|^2 
\mbox{ for all } \zeta \in \R^d, \x\in \Omega,
$$
Analogously, in the fracture network $\G$, we denote by $\Lambda_f\in L^{\infty}(\Gamma)^{(d-1)\times (d-1)}$ the 
tangential permeability tensor, and assume that there exist $\overline\lambda_f\geq \underline\lambda_f > 0$, such that
holds
$$
\underline\lambda_f|\zeta|^2 \leq (\Lambda_f(\x)\zeta,\zeta) \leq \overline\lambda_f|\zeta|^2 
\mbox{ for all } \zeta \in \R^{d-1}, \x\in\Gamma. 
$$
At the fracture network $\G$, we introduce the orthonormal system \linebreak
$(\bm\tau_1(\x),\bm\tau_2(\x),\n(\x))$, defined a.e. on $\G$. 
Inside the fractures, the normal direction is assumed to be a permeability principal direction. The normal permeability 
$\lambda_{f,\n}  \in L^{\infty}(\Gamma)$ is such that 
$\underline \lambda_{f,\n} \leq  \lambda_{f,\n}(\x) \leq  \overline \lambda_{f,\n}$ for a.e. $\x\in \Gamma$ with  
$0 < \underline \lambda_{f,\n} \leq  \overline \lambda_{f,\n}$. 
We also denote by $d_f \in L^\infty(\G)$ 
the width of the fractures assumed to be such that there exist 
${\overline d}_f\geq {\underline d}_f > 0$ with 
$$
{\underline d}_f \leq d_f(\x) \leq {\overline d}_f
$$
for a.e. $\x\in\Gamma$. 
Let us define the weighted 
Lebesgue $d-1$ dimensional measure on $\G$ 
by  $d\tau_f(\x) = d_f(\x) \d\tau(\x)$. 
We consider the source terms $h_m\in L^2(\Omega)$ (resp. $h_f\in L^2(\G)$) 
in the matrix domain $\Omega\setminus\overline\G$ (resp. in the fracture network $\G$).
The half normal transmissibility in the fracture network is denoted by 
$T_f = \frac{2\lambda_{f,\n}}{d_f}$. \\

Given $\xi\in ({1\over 2},1]$, the PDEs model writes: find $(u_m,u_f)\in V^0$, 
$(\q_m, \q_f) \in W$ such that: \\
\begin{eqnarray}
\label{modeleCont}
\left\{\begin{array}{r@{\,\,}c@{\,\,}ll}
\div(\q_m) &=& h_m &\mbox{ on } \Omega\setminus \overline \G,\\
\q_m &=& -\Lambda_m \nabla u_m &\mbox{ on } \Omega\setminus \overline \G,\\  
\gamma_{\n,\alpha^\pm(i)} \q_m &=& \frac{T_f}{2\xi - 1} (\xi \gamma_{\alpha^\pm(i)} u_m + (1-\xi) \gamma_{\alpha^\mp(i)}u_m - u_f) &\mbox{ on } \G_i,\ i\in I,\\
\div_{\tau_i}(\q_f)  - \gamma_{\n,\alpha^+(i)}\q_m - \gamma_{\n,\alpha^-(i)}\q_m &=& d_f h_f &\mbox{ on } \G_i, i\in I\\
\q_f &=& -d_f ~\Lambda_f \nabla_\tau u_f  &\mbox{ on } \G,  
\end{array}\right.
\end{eqnarray}

\subsubsection{Weak formulation}

The hybrid dimensional weak formulation amounts to 
find $(u_m,u_f)\in V^0$  satisfying the following 
variational equality for all $(v_m,v_f) \in V^0$: 
\begin{eqnarray}
\label{formVar}
\left.\begin{array}{r@{\,\,}c@{\,\,}ll}
&&\dsp \int_\Omega  \Lambda_m\nabla u_m \cdot\nabla v_m \d\x  
+ \dsp\int_\G  \Lambda_f  \nabla_\tau u_f \cdot\nabla_\tau v_f \d\tau_f(\x)  \\
&& +\dsp \sum_{i\in I}\int_{\G_i} {T_f\over 2\xi-1} \sum_{(\alpha,\beta)\in\{(\alpha^\pm(i),\alpha^\mp(i))\}}
\( \xi \gamma_{\alpha} u_m + (1-\xi) \gamma_{\beta} u_m - u_f\)\(\gamma_{\alpha} v_m - v_f\) \d\tau(\x) \\
&&\dsp - \int_\O h_m v_m \d\x 
- \int_\G h_f v_f \d\tau_f(\x) = 0. 
\end{array}\right.
\end{eqnarray}
The following proposition states the well posedness of the variational formulation \eqref{formVar}. 
\begin{proposition}
For all $\xi\in ({1\over 2},1]$, the variational problem \eqref{formVar} has a unique solution $(u_m,u_f)\in V^0$ 
which satisfies the a priori estimate 
$$
\|(u_m,u_f)\|_{V^0} \leq C \( \|h_m\|_{L^2(\O)} + \|h_f\|_{L^2(\G)}\), 
$$
with $C$ depending only on $\xi$, 
$\mathcal C_P$, $\ul\lambda_m$, $\ul\lambda_f$, $\ul d_f$, $\overline d_f$, and $\ul \lambda_{f,n}$.
In addition
$(\q_m,\q_f) =   \linebreak - (\Lambda_m \nabla u_m, d_f \Lambda_f \nabla_\tau u_f)$ belongs to $W$. 
\end{proposition}

\begin{proof}
Using that for all $\xi\in ({1\over 2},1]$ and for all $(a,b)\in \R^2$ one has 
$$
a^2 + b^2 \leq (\xi a + (1-\xi)b)a + (\xi b + (1-\xi)a)b  
\leq {1\over 2\xi-1}(a^2 + b^2),
$$
the Lax-Milgram Theorem applies, which ensures the statement of the proposition.
\hfill\qed
\end{proof}

\section{Gradient Discretization of the Hybrid Dimensional Model}
\label{sec_GS}
\subsection{Gradient Scheme Framework}

A gradient discretization $\D$ of hybrid dimensional Darcy flow models is defined by a vector space of 
degrees of freedom $X_{\D} = X_{\D_m}\times X_{\D_f}$, its subspace satisfying ad hoc homogeneous 
boundary conditions  $X^0_{\D} = X^0_{\D_m}\times X^0_{\D_f}$, 
and the following gradient and reconstruction operators: 
\begin{itemize}
\item Gradient operator on the matrix domain: 
$\nabla_{\D_m} : X_{\D_m} \rightarrow L^2(\O)^d$ 
\item Gradient operator on the fracture network: $\nabla_{\D_f} : X_{\D_f} \rightarrow L^2(\G)^{d-1}$
\item A function reconstruction operator  on the matrix domain: \\
$\Pi_{{\D_m}} : X_{{\D_m}} \rightarrow L^2(\O)$
\item Two function reconstruction operators on the fracture network: \\
$\Pi_{{\D_f}} : X_{{\D_f}} \rightarrow L^2(\G)$ 
and $\wt \Pi_{{\D_f}} : X_{{\D_f}} \rightarrow L^2(\G)$
\item Reconstruction operators of the trace on $\G_\alpha$ for $\alpha\in\chi$:  \\
$\Pi^\alpha_{\D_m} : X_{{\D_m}} \rightarrow L^2(\G_\alpha)$.
\end{itemize}

The space  $X_{\D}$ is endowed with the seminorm 
\begin{align*}
\|(v_{\D_m},v_{\D_f})\|_{\D} &= \dsp \(\|\nabla_{\D_m} v_{\D_m}\|^2_{L^2(\O)^d} + \|\nabla_{\D_f} v_{\D_f}\|^2_{L^2(\G)^{d-1}} 
 + \dsp \sum_{\alpha\in\chi} \| \Pi_{\D_m}^\alpha v_{\D_m} - \wt \Pi_{\D_f} v_{\D_f} \|_{L^2(\G_\alpha)}^2\)^{1\over 2},
\end{align*}
which is assumed to define a norm on $X_{\D}^0$. \\

The following properties of gradient discretizations are crucial for the convergence analysis of the corresponding numerical schemes:\\

\noindent {\bf Coercivity}: Let $\D$ be a gradient discretization and
$$
\mathcal C_\D = \max_{0\neq(v_{\D_m},v_{\D_f})\in X_\D^0}\frac{\|\Pi_{\D_m} v_{\D_m}\|_{L^2(\O)} + \|\Pi_{\D_f} v_{\D_f}\|_{L^2(\G)}}{\|(v_{\D_m},v_{\D_f})\|_{\D}}. 
$$
A sequence $(\D^l)_{l\in\N}$ of gradient discretizations is said to be coercive, if there exists $\overline{\mathcal{C}}_P > 0$ such that
$\mathcal C_{\D^l}\leq \overline{\mathcal C}_P$ for all $l\in\N$.\\

\noindent {\bf Consistency}: Let $\D$ be a gradient discretization. For $u = (u_m,u_f)\in V^0$ and $v_\D = (v_{\D_m},v_{\D_f})\in X_{\D}^0$ let us define 
$$
\left.\begin{array}{r@{\,\,}c@{\,\,}ll}
s(v_\D,u) &=&\|\nabla_{\D_m} v_{\D_m} -\nabla u_m\|_{L^2(\O)^d} + \|\nabla_{\D_f} v_{\D_f} -\nabla_\tau u_f\|_{L^2(\G)^{d-1}} \\
&+& \|\Pi_{\D_m} v_{\D_m} -u_m\|_{L^2(\O)} + \|\Pi_{\D_f} v_{\D_f} - u_f\|_{L^2(\G)} \\
&+&  \|\wt \Pi_{\D_f} v_{\D_f} - u_f\|_{L^2(\G)}  + 
\sum_{\alpha\in\chi}\|\Pi_{\D_m}^\alpha v_{\D_m} -\gamma_\alpha u_m\|_{L^2(\G_\alpha)}. 
\end{array}\right.
$$
and ${\cal S}_\D(u) = \min_{v_{\D}\in X_{\D}^0} s(v_\D, u)$. 
A sequence $(\D^l)_{l\in\N}$ of gradient discretizations is said to be consistent, if for all $u = (u_m,u_f)\in V^0$ holds
$$
\lim_{l \rightarrow \infty} {\cal S}_{\D^l}(u) = 0. 
$$

\noindent {\bf Limit Conformity}: Let $\D$ be a gradient discretization. For all $\q = (\q_m, \q_f) \in W,\ v_\D = (v_{\D_m},v_{\D_f})$ we define 
$$
\left.\begin{array}{r@{\,\,}c@{\,\,}ll}
w(v_\D,\q) &=& \dsp \int_\O \(\nabla_{\D_m} v_{\D_m} \cdot \q_m + (\Pi_{\D_m} v_{\D_m}) \div\q_m \) \d\x \\
&+& \dsp \int_{\G} \( \nabla_{\D_f} v_{\D_f} \cdot \q_f 
+ (\Pi_{\D_f} v_{\D_f}) \div_\tau\q_f \) \d\tau(\x) \\
&+& \dsp \sum_{\alpha\in\chi}\int_{\G_\alpha} \gamma_{\n,\alpha} \q_m \( \wt \Pi_{\D_f} v_{\D_f} - \Pi_{\D_f} v_{\D_f} - \Pi_{\D_m}^\alpha v_{\D_m}\) \d\tau(\x)
\end{array}\right.
$$
and ${\cal W}_\D(\q) = \max_{0\neq v_{\D}\in X_{\D}^0}\frac{1}{\|v_\D\|_\D} | w(v_\D,\q)|$.
A sequence $(\D^l)_{l\in\N}$ of gradient discretizations is said to be limit conforming, if for all $\q = (\q_m, \q_f) \in W$ holds
$$
\lim_{l \rightarrow \infty} {\cal W}_{\D^l}(\q) = 0. 
$$

\begin{lemma}
\label{lemmakonstantin}
Let $(\D^l)_{l\in\N} = (X_{\D^l}^0, \Pi_{\D_m^l}, \Pi_{\D_f^l}, \wt\Pi_{\D_f^l}, (\Pi_{\D_m^l}^\alpha)_{\alpha\in\chi}, \grad_{\D_m^l}, \grad_{\D_f^l})_{l\in\N}$ and\\
$(\ov\D^l)_{l\in\N} = (X_{\D^l}^0, \ov\Pi_{\D_m^l}, \ov\Pi_{\D_f^l}, \wt{\ov\Pi}_{\D_f^l}, (\ov\Pi_{\D_m^l}^\alpha)_{\alpha\in\chi}, \grad_{\D_m^l}, \grad_{\D_f^l})_{l\in\N}$ be two sequences of gradient discretisations of \eqref{formVar} and let us assume that $(\D^l)_{l\in\N}$ is coercive, consistent and limit conforming.
Let us furthermore assume that the sequence $(\zeta_{\D^l,\ov\D^l})_{l\in\N}$, defined by
\begin{align*}
\zeta_{\D^l,\ov\D^l} &:= \max_{0\neq v_{\D^l}\in X_{\D^l}^0} \(
\frac{1}{\|v_{\D^l} \|_{\D^l}}\cdot\(\|\Pi_{\D_m^l} v_{\D_m^l} - \ov\Pi_{\D_m^l} v_{\D_m^l}\|_{L^2(\O)} 
+ \|\Pi_{\D_f^l} v_{\D_f^l} - \ov\Pi_{\D_f^l} v_{\D_f^l}\|_{L^2(\G)} \\
&+  \|\wt \Pi_{\D_f^l} v_{\D_f^l} - \wt {\ov\Pi}_{\D_f^l} v_{\D_f^l}\|_{L^2(\G)}  
+ \sum_{\alpha\in\chi}\|\Pi_{\D_m^l}^\alpha v_{\D_m^l} -\ov\Pi_{\D_m^l}^\alpha v_{\D_m^l}\|_{L^2(\G_\alpha)}\)\),
\end{align*}
satisfies 
\begin{align}
\begin{aligned}
\label{eqlemmakonstantin0}
\lim_{l \to \infty}\zeta_{\D^l,\ov\D^l} = 0
\end{aligned}
\end{align}
 and that there is a constant $C\in\R$ independent of $l$ such that
\begin{align}
\label{eqlemmakonstantin1}
\sum_{\alpha\in\chi} \|\Pi_{\D_m^l}^\alpha v_{\D_m^l} - \wt \Pi_{\D_f^l} v_{\D_f^l} \|_{L^2(\G_\alpha)}
\leq C\cdot \sum_{\alpha\in\chi} \| \ov\Pi_{\D_m^l}^\alpha v_{\D_m^l} - \wt{\ov\Pi}_{\D_f^l} v_{\D_f^l} \|_{L^2(\G_\alpha)}
\end{align}
for all $v_{\D^l}\in X^0_{\D^l},\ l\in\N$. Then $(\ov\D^l)_{l\in\N}$ is coercive, consistent and limit conforming.
\end{lemma}

\begin{proof}

\noindent\textit{Coercivity: } $(\ov\D^l)_{l\in\N}$ is coercive, since for all $l\in\N$ we have (with $\D = \D^l, \ov\D = \ov\D^l$)
\begin{align*}
\|\ov\Pi_{\D_m} v_{\D_m}\|_{L^2(\O)} + \|\ov\Pi_{\D_f} v_{\D_f}\|_{L^2(\G)}
\leq (\zeta_{\D,\ov\D} + \mathcal C_\D)\|v_{\D}\|_\D
\leq \max(1,C) (\zeta_{\D,\ov\D} + \mathcal C_\D)\|v_{\D}\|_{\ov\D}
\end{align*}
and since $\max(1,C)\cdot(\zeta_{\D^l,\ov\D^l} + \mathcal C_{\D^l})$ is uniformly bounded. In the last inequality we have used that 
$\|v_\D\|_\D\leq\max(1,C)\|v_\D\|_{\ov\D}$, which follows from \eqref{eqlemmakonstantin1}. \\

\noindent\textit{Consistency: } Let $l\in\N$ be fixed and $\D = \D^l, \ov\D = \ov\D^l$. We first choose, for a given $u = (u_m,u_f)\in V^0$, a $\ul v_\D\in X_\D^0$, such that $s_\D(\ul v_\D, u) = {\cal S}_\D(u)$.
Using the inequality
\begin{align*}
s_{\ov\D}(v_\D, u)&\leq s_{\D}(v_\D, u) + \|\Pi_{\D_m} v_{\D_m} - \ov\Pi_{\D_m} v_{\D_m}\|_{L^2(\O)} 
+ \|\Pi_{\D_f} v_{\D_f} - \ov\Pi_{\D_f} v_{\D_f}\|_{L^2(\G)} \\
&+  \|\wt \Pi_{\D_f} v_{\D_f} - \wt {\ov\Pi}_{\D_f} v_{\D_f}\|_{L^2(\G)}  
+ \sum_{\alpha\in\chi}\|\Pi_{\D_m}^\alpha v_{\D_m} -\ov\Pi_{\D_m}^\alpha v_{\D_m}\|_{L^2(\G_\alpha)},
\end{align*}
which holds for all $v_\D\in X_\D$, we obtain
$$
{\cal S}_{\ov\D}(u)\leq{\cal S}_{\D}(u) + \zeta_{\D,\ov\D}\|\ul v_\D\|_\D.
$$
Moreover
$$
\|\ul v_\D\|_\D\leq {\cal S}_\D(u) + \|\grad u_m\|_{L^2(\O)^d} + \|\grad_\tau u_f\|_{L^2(\G)^{d-1}} 
+ \sum_{\alpha\in\chi} \| (\gamma_{\alpha} u_m - u_f)\|_{L^2(_{\G_\alpha})},
$$
which implies that $\|\ul v_{\D^l}\|_{\D^l}$ is uniformly bouded and therefore ${\cal S}_{\ov\D^l}(u)\to 0 $ as $l\to\infty$.

\noindent\textit{Limit Conformity: } Let again $l\in\N$ be fixed and $\D = \D^l, \ov\D = \ov\D^l$. For given $\q = (\q_m,\q_f)\in W$ and $v_\D\in X_\D^0$ we calculate
\begin{align*}
w_{\ov\D}(v_\D,\q) &\leq w_\D(v_\D,\q) + \|\Pi_{\D_m} v_{\D_m} - \ov\Pi_{\D_m} v_{\D_m}\|_{L^2(\O)} 
\cdot\|\div\q_m\|_{L^2(\O)}\\
&+ \|\Pi_{\D_f} v_{\D_f} - \ov\Pi_{\D_f} v_{\D_f}\|_{L^2(\G)} \cdot\|\div_\tau\q_f\|_{L^2(\G)}
 +  \sum_{\alpha\in\chi}\( \|\Pi_{\D_f} v_{\D_f} - {\ov\Pi}_{\D_f} v_{\D_f}\|_{L^2(\G_\alpha)}  \\
&+ \|\wt \Pi_{\D_f} v_{\D_f} - \wt {\ov\Pi}_{\D_f} v_{\D_f}\|_{L^2(\G_\alpha)}  
+ \|\Pi_{\D_m}^\alpha v_{\D_m} -\ov\Pi_{\D_m}^\alpha v_{\D_m}\|_{L^2(\G_\alpha)} \)\cdot\|\trace_{\n,\alpha}\q_m\|_{L^2(\G_\alpha)}\\
&\leq w_\D(v_\D, \q) + \zeta_{\D,\ov\D}\cdot\|v_\D\|_\D\cdot\( \|\div\q_m\|_{L^2(\O)} 
 + \|\div_\tau\q_f\|_{L^2(\G)} + \sum_{\alpha\in\chi} \|\trace_{\n,\alpha}\q_m\|_{L^2(\G_\alpha)} \).
\end{align*}
Taking \eqref{eqlemmakonstantin1} into account, we derive
$$
{\cal W}_{\ov\D}(\q_m, \q_f)\leq \max(1,C)\cdot\sup_{0\neq v_\D\in X_\D}\frac{w_{\ov\D}(v_\D,\q)}{\|v_\D\|_\D} 
\leq\max(1,C)\cdot( {\cal W}_{\D}(\q_m, \q_f) + \zeta_{\D,\ov\D}\|\q\|_W ).
$$
Therefore ${\cal W}_{\ov\D^l}(\q_m, \q_f)$ tends to zero as $l$ goes to infinity.
\hfill\qed
\end{proof}

\begin{proposition}\label{theoremlimitregularity}\emph{(Regularity at the Limit)}
Let $(\D^l)_{l\in \N}$ be a coercive and limit conforming sequence of gradient discretizations 
and let $(v_{\D_m^l},v_{\D_f^l})_{l\in \N}$ be a uniformly bounded sequence in 
$X^0_{\D^l}$. Then, there exist $(v_m,v_f)\in V^0$ and a subsequence still denoted by $(v_{\D_m^l},v_{\D_f^l})_{l\in \N}$ 
such that 
$$
\left\{\begin{array}{r@{\,\,}c@{\,\,}l}
&&\Pi_{\D_m^l}v_{\D_m^l} \rightharpoonup v_m \quad\mbox{ in } L^2(\O),\\
&&\nabla_{\D_m} v_{\D_m^l} \rightharpoonup \nabla v_m \quad\mbox{ in } L^2(\O)^d,\\
&&\Pi_{\D_f^l}v_{\D_f^l} \rightharpoonup v_f \quad\mbox{ in } L^2(\G),\\
&&\nabla_{\D_f} v_{\D_f^l} \rightharpoonup \nabla_\tau v_f \quad\mbox{ in } L^2(\G)^{d-1},\\
&&\wt \Pi_{\D_f} v_{\D_f^l} -\Pi^\alpha_{\D_m} v_{\D_m^l}  \rightharpoonup v_f - \gamma_\alpha v_m\quad\mbox{ in } L^2(\G_\alpha),\text{ for all }\alpha\in\chi.\\
\end{array}\right.
$$
\end{proposition}

\begin{proof}
By definition of the norm of $X_{\D^l}^0$ and by coercivity, $\Pi_{\D_m^l}v_{\D_m^l},\ \Pi_{\D_f^l}v_{\D_f^l},\ \nabla_{\D_m} v_{\D_m^l},\ \nabla_{\D_f} v_{\D_f^l}$ and $(\Pi^\alpha_{\D_m} u_{\D_m^l} - \wt \Pi_{\D_f} u_{\D_f^l}),\ \alpha\in\chi$,
are uniformly bounded in $L^2$ (for $l\to\infty$). Therefore there exist 
$v_m\in L^2(\O),\ v_f\in  L^2(\G),\ G\in L^2(\O)^d,\ H\in L^2(\G)^{d-1}$ and 
$J_\alpha\in L^2(\G_\alpha),\ \alpha\in\chi$, and a subsequence still denoted by $(v_{\D_m^l},v_{\D_f^l})_{l\in \N}$ 
such that 
$$
\begin{array}{r@{\,\,}c@{\,\,}l}
&&\Pi_{\D_m^l}v_{\D_m^l} \rightharpoonup v_m \quad\mbox{ in } L^2(\O),\\
&&\nabla_{\D_m} v_{\D_m^l} \rightharpoonup G \quad\mbox{ in } L^2(\O)^d,\\
&&\Pi_{\D_f^l}v_{\D_f^l} \rightharpoonup v_f \quad\mbox{ in } L^2(\G),\\
&&\nabla_{\D_f} v_{\D_f^l} \rightharpoonup H \quad\mbox{ in } L^2(\G)^{d-1},\\
&&\wt \Pi_{\D_f} v_{\D_f^l} - \Pi^\alpha_{\D_m} v_{\D_m^l} \rightharpoonup J_\alpha\quad\mbox{ in } L^2(\G_\alpha),
\text{ for }\alpha\in\chi.\\
\end{array}
$$
Using limit conformity we obtain (by letting $l\to\infty$)
\begin{align}
\label{prooftheoremlimitregularity_1}
\int_\O(G\cdot\q_m + v_m\div\q_m) \d\x + \int_\G (H\cdot\q_f + v_f\div_\tau\q_f) \d\tau(\x)
+ \sum_{\alpha\in\chi}\int_{\G_\alpha}\trace_{\n,\alpha}\q_m
(J_\alpha - v_f) \d\tau(\x) = 0
\end{align}
for all $(\q_m, \q_f)\in C^\infty_{W_m}\times C^\infty_{W_f}$.
The statement of the proposition follows now from Lemma \ref{lemmaweakderivatives}.
\hfill\qed
\end{proof}

\begin{corollary}
\label{cortheoremlimitregularity}
Let $(\D^l)_{l\in \N}$ be a sequence of gradient discretizations, assumed to be limit conforming against regular test functions $(\q_m, \q_f)\in C^\infty_{W_m}\times C^\infty_{W_f}$ and let $(v_{\D_m^l},v_{\D_f^l})_{l\in \N}$ be a uniformly bounded sequence in 
$X^0_{\D^l}$, such that $\Pi_{\D_m^l}v_{\D_m^l}$ and $\Pi_{\D_f^l}v_{\D_f^l}$ are uniformly bounded in $L^2$ (for $l\rightarrow\infty$). Then holds the conclusion of Proposition \ref{theoremlimitregularity}.
\end{corollary}

\subsection{Application to \eqref{formVar}}

The non conforming discrete variational formulation of the model problem is defined by: find 
$(u_{\D_m},u_{\D_f})\in X_{\D}^0$ such that 
\begin{eqnarray}
\label{GradientScheme}
\left.\begin{array}{r@{\,\,}c@{\,\,}ll}
&&\dsp \int_\Omega  \Lambda_m\nabla_{\D_m} u_{\D_m} \cdot\nabla_{\D_m} v_{\D_m} \d\x   
+ \dsp\int_\G  \Lambda_f  \nabla_{\D_f} u_{\D_f} \cdot\nabla_{\D_f} v_{\D_f} \d\tau_f(\x)  
+\dsp \sum_{i\in I}\int_{\G_i} {T_f\over 2\xi-1} \\
&& \dsp \sum_{(\alpha,\beta)\in\{(\alpha^\pm(i),\alpha^\mp(i))\}} 
\!\!\!\!\!\!\!\!\!\!\!\!\!\!\!\!
\( \xi \Pi_{\D_m}^{\alpha} u_{\D_m} + (1-\xi) \Pi_{\D_m}^{\beta} u_{\D_m} - \wt \Pi_{\D_f} u_{\D_f}\)
\(\Pi_{\D_m}^{\alpha} v_{\D_m} - \wt \Pi_{\D_f} v_{\D_f}\) \d\tau(\x) \\
&&\dsp - \int_\O h_m \Pi_{\D_m} v_{\D_m} \d\x 
- \int_\G h_f \Pi_{\D_f} v_{\D_f} \d\tau_f(\x) = 0, 
\end{array}\right.
\end{eqnarray}
for all $(v_{\D_m},v_{\D_f})\in X_{\D}^0$. 

\begin{proposition}
Let $\xi\in ({1\over 2},1]$ and $\D$ be a gradient discretization, then 
\eqref{GradientScheme} has a unique solution $(u_{\D_m},u_{\D_f})\in X_\D^0$ satisfying the a priori estimate 
$$
\|(u_{\D_m},u_{\D_f})\|_\D \leq C \( \|h_m\|_{L^2(\O)} + \|h_f\|_{L^2(\G)}\)
$$
with $C$ depending only on $\xi$, $\mathcal C_\D$, $\ul\lambda_m$, $\ul\lambda_f$, $\ul d_f$, $\overline d_f$, and $\ul \lambda_{f,n}$. 
\end{proposition}

\begin{proof}
The Lax-Milgram Theorem applies, which ensures this result.
\hfill\qed
\end{proof}

The main theoretical result for gradient schemes is stated by the following proposition:

\begin{proposition}\emph{(Error Estimate)}
\label{properror}
Let $u = (u_m,u_f)\in V^0$, $\q = (\q_m,\q_f)\in W$ be the solution of \eqref{modeleCont}. 
Let $\xi\in ({1\over 2},1]$, $\D$ be a gradient discretization and $u_\D = (u_{\D_m},u_{\D_f})\in X_\D^0$ be the solution of \eqref{GradientScheme}. 
Then, there exists $C_0 > 0$ depending only on $\xi$, $\mathcal C_\D$, $\ul\lambda_m$, $\ul\lambda_f$,$\ov\lambda_m$, $\ov\lambda_f$, 
$\ul d_f$, $\ov d_f$, $\ul \lambda_{f,n}$, and $\ov \lambda_{f,n}$ such that one has the following error estimate:  
\begin{align*}
&\|\Pi_{\D_m} u_{\D_m} -u_m\|_{L^2(\O)} + \|\Pi_{\D_f} u_{\D_f} - u_f\|_{L^2(\G)} 
+  \dsp\|\wt \Pi_{\D_f} u_{\D_f} - u_f\|_{L^2(\G)}  \\
& + \sum_{\alpha\in\chi}\|\Pi_{\D_m}^\alpha u_{\D_m} -\gamma_\alpha u_m\|_{L^2(\G_\alpha)}
+ \|\nabla u_m-\nabla_{\D_m} u_{\D_m}\|_{L^2(\O)^d} + \|\nabla_\tau u_f-\nabla_{\D_f} u_{\D_f}\|_{L^2(\G)^{d-1}} \\
&\leq C_0 (\mathcal{S}_{\D}(u_m,u_f) + \mathcal{W}_{\D}(\q_m,\q_f)).
\end{align*}
\end{proposition}

\begin{proof}
From the definition of $\mathcal{W}_\D$, and using the definitions 
\eqref{modeleCont} of the solution $u,\q$ and \eqref{GradientScheme} 
of the discrete solution $u_\D$, 
it holds for all $(v_{\D_m},v_{\D_f})\in V^0$
\begin{align}
\begin{aligned}
\label{prooferrorestimate1}
&\|(v_{\D_m},v_{\D_f})\|_\D\cdot \mathcal{W}_\D(\q_m,\q_f)\\
&\geq \bigg\vert \int_\O  \(\nabla_{\D_m} v_{\D_m} \cdot \q_m + (\Pi_{\D_m} v_{\D_m}) h_m \) \d\x 
+ \int_{\G} \( \nabla_{\D_f} v_{\D_f} \cdot \q_f 
+ (\Pi_{\D_f} v_{\D_f}) d_f h_f \) \d\tau(\x)\\
&+ \sum_{i\in I}\int_{\G_i} {T_f\over 2\xi-1} \sum_{(\alpha,\beta)\in\{(\alpha^\pm(i),\alpha^\mp(i))\}} 
\!\!\!\!\!\!\!\!\!\!\!\!\!\!\( \xi \gamma_{\alpha} u_m + (1-\xi) \gamma_{\beta} u_m - u_f\)
\( \wt \Pi_{\D_f} v_{\D_f}- \Pi_{\D_m}^{\alpha} v_{\D_m}\) \d\tau(\x)\bigg\vert\\
&= \bigg\vert \int_\O  \( \Lambda_m\nabla_{\D_m} v_{\D_m} \cdot ( \nabla_{\D_m} u_{\D_m} - \grad u_m ) \) \d\x 
+ \int_{\G} \( \Lambda_f\nabla_{\D_f} v_{\D_f} \cdot ( \nabla_{\D_f} u_{\D_f} - \grad_\tau u_f ) \) \d\tau_f(\x)\\
&+ \sum_{i\in I}\int_{\G_i} {T_f\over 2\xi-1} \sum_{(\alpha,\beta)\in\{(\alpha^\pm(i),\alpha^\mp(i))\}} 
\!\!\!\!\!\!\!\!\!\!\!\!\!\! \( \wt \Pi_{\D_f} v_{\D_f} - \Pi_{\D_m}^\alpha v_{\D_m}\) \\
& \quad\quad\quad\quad \times\(\xi \gamma_\alpha u_m + (1-\xi)\gamma_\beta u_m - u_f 
- \xi \Pi_{\D_m}^\alpha u_{\D_m} - (1-\xi) \Pi_{\D_m}^\beta u_{\D_m} + \wt \Pi_{\D_f} u_{\D_f} \)  \d\tau(\x)\bigg\vert
\end{aligned}
\end{align}
Let us choose $w_\D = (w_{\D_m},w_{\D_f}) \in X_\D^0$, s.t. $s(w_\D,u) = \mathcal S_\D(u)$ and 
set $(v_{\D_m},v_{\D_f}) = u_\D - w_\D$ in \eqref{prooferrorestimate1}. Then holds
\begin{align*}
&\|\nabla u_m-\nabla_{\D_m} u_{\D_m}\|_{L^2(\O)^d} + \|\nabla_\tau u_f-\nabla_{\D_f} u_{\D_f}\|_{L^2(\G)^{d-1}}  \\
& + \sum_{\alpha\in\chi}\|\Pi_{\D_m}^\alpha u_{\D_m} - \wt \Pi_{\D_f} u_{\D_f} -\gamma_\alpha u_m + u_f\|_{L^2(\G_\alpha)}
\leq C \cdot (\mathcal{S}_{\D}(u_m,u_f) + \mathcal{W}_{\D}(\q_m,\q_f)),
\end{align*}
with a constant $C > 0$ depending only on $\xi$, $\ul\lambda_m$, $\ul\lambda_f$,$\ov\lambda_m$, $\ov\lambda_f$, 
$\ul d_f$, $\ov d_f$, $\ul \lambda_{f,n}$, and $\ov \lambda_{f,n}$. Taking coercivity into account leads to the statement of the proposition.
\hfill\qed
\end{proof}

\section{Two Examples of Gradient Schemes}
\label{sec_VAGHFV}

Following \cite{Eymard.Herbin.ea:2010}, we consider generalised polyhedral meshes of $\Omega$.
Let $\cells$ be the set of cells that are disjoint open subsets of $\Omega$ such that
$\bigcup_{K\in\cells} \ov K = \overline\Omega$. For all $K\in\cells$, ${\x}_K$ denotes the so-called ``center'' of the cell $K$ under the assumption that $K$ is star-shaped with respect to ${\x}_K$. 
Let $\faces$ denote the set of faces of the mesh. The faces are not assumed to be planar for the VAG discretization, hence the term ``generalised polyhedral cells'', but they need to be planar for the HFV discretization. We denote by $\nodes$ the set of vertices of the mesh. Let $\nodes_K$, $\faces_K$, $\nodes_\sigma$
respectively denote the set of the vertices of $K\in\cells$, faces of $K$, and vertices of $\sigma\in \faces$. For any face $\sigma\in \faces_K$, we have $\nodes_\sigma \subset \nodes_K$.
Let $\cells_\s$ (resp. $\faces_\s$) denote the set of the cells (resp. faces) sharing the vertex $\s\in\nodes$. 
The set of edges of the mesh is denoted by $\edges$ 
and $\edges_\sigma$ denotes the set of edges of the face $\sigma\in \faces$. Let 
$\faces_e$ denote the set of faces sharing the edge $e\in {\cal E}$, and ${\cal M}_\sigma$ denote the set of cells 
sharing the face $\sigma\in\faces$.  
We denote by ${\cal F}_{ext}$ the subset of faces $\sigma \in \faces $ such that $\cells_\sigma$ has only one element, and 
we set $\edges_{ext} = \bigcup_{\sigma\in \faces_{ext}} \edges_\sigma$, and  
$\nodes_{ext} = \bigcup_{\sigma\in \faces_{ext}} \nodes_\sigma$. The mesh is assumed to be conforming 
in the sense that for all $\sigma\in \faces\sm\faces_{ext}$, the set $\cells_\sigma$ contains exactly two cells. 
It is assumed that for each face $\sigma\in\faces$, there exists a so-called ``center'' of the face
${\x}_\sigma$ such that
$$
{\x}_\sigma = \sum_{\s\in \nodes_\sigma} \beta_{\sigma,\s}~\x_\s, \mbox{ with }
\sum_{\s\in \nodes_\sigma} \beta_{\sigma,\s}=1,
$$
where $\beta_{\sigma,\s}\geq 0$ for all $\s\in \nodes_\sigma$.
The face $\sigma$ is assumed to match with the union of the triangles
$T_{\sigma,e}$ defined by the face center ${\x}_\sigma$
and each of its edge $e\in\edges_\sigma$. \\

The mesh is assumed to be conforming w.r.t. the fracture network $\G$ in the sense 
that there exist subsets $\faces_{\G_i}$, $i\in I$ of $\faces$ such that 
\begin{align}
\label{eqconfmesh}
\overline \G_i = \bigcup_{\sigma\in\faces_{\G_i}} \bar\sigma.  
\end{align}
We will denote by $\faces_\G$ the set of fracture faces $\bigcup_{i\in I} \faces_{\G_i}$.
Similarly, we will denote by $\edges_\G$ the set of fracture edges $\bigcup_{\sigma\in\faces_\G}\edges_\sigma$ and by $\nodes_\G$ the set of fracture vertices $\bigcup_{\sigma\in\faces_\G} \nodes_\sigma$.\\

We also define a submesh $\T$ of tetrahedra, where each tetrahedron $D_{K,\sigma,e}$ is the convex hull of the 
cell center $\x_K$ of $K$, the face center $\x_\sigma$ of $\sigma\in\faces_K$ and the edge $e\in \edges_\sigma$. 
Similarly we define a triangulation $\Delta$ of $\G$, such that we have:
$$
\T = \bigcup_{K\in\faces,\sigma\in\faces_K,e\in\edges_\sigma}D_{K,\sigma,e}\quad\text{and}\quad
\Delta=\bigcup_{\sigma\in\faces_\G,e\in\edges_\sigma}T_{\sigma,e}.
$$
We introduce for $D\in\T$ the diameter $h_D$ of $D$ and set $h_\T = \max_{D\in\T}h_D$. 
The regularity of our polyhedral mesh will be measured by the shape regularity of the tetrahedral submesh defined by 
$\theta_\T = \max_{D\in\T}\frac{h_D}{\rho_D}$ 
where $\rho_D$ is the insphere diameter of $D\in\T$.\\

The set of matrix $\times$ fracture degrees of freedom is denoted by $\dofm\times\doff$.
The real vector spaces $X_{\D_m}$ and $X_{\D_f}$ of discrete unknowns in the matrix and in the 
fracture network respectively are then defined by
\begin{align*}
X_{\D_m} & = \operatorname{span}\{\mathfrak e_\nu\mid\nu\in\dofm\}\\
X_{\D_f} & = \operatorname{span}\{\mathfrak e_\nu\mid\nu\in\doff\},
\end{align*}
where
$$
\mathfrak e_\nu = \left\{
\begin{array}{l l}
(\delta_{\nu\mu})_{\mu\in\dofm} &\quad \text{for }\nu\in\dofm\\
(\delta_{\nu\mu})_{\mu\in\doff} &\quad \text{for }\nu\in\doff.
\end{array}\right.
$$
For $u_{\D_m}\in X_{\D_m}$ and $\nu\in\dofm$ we denote by $u_\nu$ the $\nu$th component of $u_{\D_m}$ and likewise for 
$u_{\D_f}\in X_{\D_f}$ and $\nu\in\doff$.
We also introduce the product of these vector spaces
$$
X_\D = X_{\D_m}\times X_{\D_f},
$$
for which we have $\operatorname{dim} X_\D = \#\dofm + \#\doff$.

To account for our homogeneous boundary conditions on $\partial \Omega$ and $\Sigma_0$ we 
introduce the subsets $\dofDirm\subset \dofm$, and $\dofDirf\subset \doff$, and we set $\dofDir = \dofDirm\times\dofDirf$, and 
$$
X_\D^0 = \{ u\in X_\D \, |\, u_\nu = 0 \mbox{ for all } \nu\in \dofDir\}.  
$$

\subsection{Vertex Approximate Gradient Discretization}
\label{sec_VAG}
In this subsection, 
the VAG discretization introduced in \cite{Eymard.Herbin.ea:2010} for diffusive problems 
on heterogeneous anisotropic media is extended to the hybrid dimensional model. 
We consider the ${\mathbb P}_1$ finite element construction as well as a finite volume version 
using lumping both for the source terms and the matrix fracture fluxes. \\

We first establish an equivalence relation on each $\cells_s,\ s\in\nodes,$ by
\begin{align*}
K\equiv_{\cells_s} L\quad \Longleftrightarrow\quad
&\mbox{ there exists }n\in\N\mbox{ and a sequence }(\sigma_i)_{i=1,\dots,n}\mbox{ in } \faces_s\backslash\faces_\G,\\
&\mbox{ such that }K\in\cells_{\sigma_1},L\in\cells_{\sigma_n}\mbox{ and }\cells_{\sigma_{i+1}}\cap\cells_{\sigma_i}
\neq\emptyset\\
&\mbox{ for }i=1,\dots,n-1.
\end{align*}
Let us then denote by $\ov\cells_s$ the set of all classes of equivalence of $\cells_s$ and by $\ov K_s$ the element of 
$\ov\cells_s$ containing $K\in\cells$. Obviously $\ov\cells_s$ might have more than one element only if $s\in\nodes_\G$.
Then we define (cf. figure \ref{figureDofVag})
\begin{align*}
\dofm & = \cells\cup\Bigl\{K_\sigma\bigm | \sigma\in\faces_\G,K\in\cells_\sigma\Bigr\}
\cup\Bigl\{\ov K_s\bigm | s\in\nodes,\ov K_s \in\ov\cells_s\Bigr\}, \\
\doff & = \faces_\G\cup\nodes_\G,\\
\dofDirm & : = \Bigl\{\ov K_s\bigm  |s\in\nodes_{ext},\ov K_s \in\ov\cells_s\Bigr\},\\
\dofDirf & = \nodes_\Gamma\cap \nodes_{ext}. 
\end{align*}
We thus have
\begin{align}
\begin{aligned}
\label{XVag}
X_{\D_m} &= \Bigl\{u_K\bigm |K\in\cells\Bigr\}\cup\Bigl\{u_{K_\sigma}\bigm |\sigma\in\faces_\G,K\in\cells_\sigma\Bigr\}\\
&\cup\Bigl\{u_{\ov K_s}\bigm |s\in\nodes,\ov K_s \in\ov\cells_s\Bigr\}, \\
X_{\D_f} &= \Bigl\{u_\sigma\bigm | \sigma\in\faces_\G\Bigr\}\cup\Bigl\{u_s\bigm | s\in\nodes_\G\Bigr\}.
\end{aligned}
\end{align}

Now we can introduce the piecewise affine interpolators (or reconstruction operators)
\begin{align*}
\Pi_\T \colon X_{\D_m}\longrightarrow H^1(\Omega\backslash\ov\G)
\qquad\text{ and }\qquad
\Pi_\Delta \colon X_{\D_f}\longrightarrow H^1(\G),
\end{align*}
which act linearly on $X_{\D_m}$ and $X_{\D_f}$, such that 
$\Pi_\T u_{\D_m}$ is affine on each $D_{K,\sigma,e}\in\T$  
 and satisfies on each cell $K\in\cells$
\begin{align*}
\begin{array}{l@{\,\,}l@{\,\,}rl}
&\Pi_\T u_{\D_m}(\x_K) = u_K,\\
&\Pi_\T u_{\D_m}(\x_s) = u_{\ov K_s} \qquad &\forall s\in\nodes_K,\\
&\Pi_\T u_{\D_m}(\x_\sigma) = u_{K_\sigma} \qquad &\forall \sigma\in\faces_K\cap\faces_\G ,\\
&\Pi_\T u_{\D_m}(\x_\sigma) = \sum\limits_{s\in\nodes_\sigma}\beta_{\sigma,s}u_{\ov K_s} \qquad 
&\forall\sigma\in\faces_K\backslash\faces_\G,
\end{array}
\end{align*}
while $\Pi_\Delta u_{\D_f}$ is affine on each $T_{\sigma,e}\in\Delta$ 
and satisfies for all $\nu\in dof_{\D_f}$
\begin{align*}
\Pi_\Delta u_{\D_f}(\x_\nu) = u_\nu,
\end{align*}
where $\x_\nu\in\ov\Omega$ is the grid point associated with the degree of freedom $\nu\in\dofm\cup\doff$.
The discrete gradients on $X_{\D_m}$ and $ X_{\D_f}$ are subsequently defined by
\begin{align}
\label{gradVag}
\grad_{\D_m} = \grad\Pi_\T\qquad
\text{and}\qquad\grad_{\D_f} = \grad_\tau\Pi_\Delta.
\end{align}

\begin{figure}
  \begin{minipage}[c]{0.3\textwidth}
\caption{
Cell $K$ touching a fracture face $\sigma$. Illustration of the simplices on which:
\newline Red: $\grad_{\D_m}$ is constant.
\newline Grey: $\grad_{\D_f}$ is constant.
}
\label{figureDofVag}
  \end{minipage}\hfill
    \begin{minipage}[c]{0.6\textwidth}
    \includegraphics[width=0.6\textwidth]{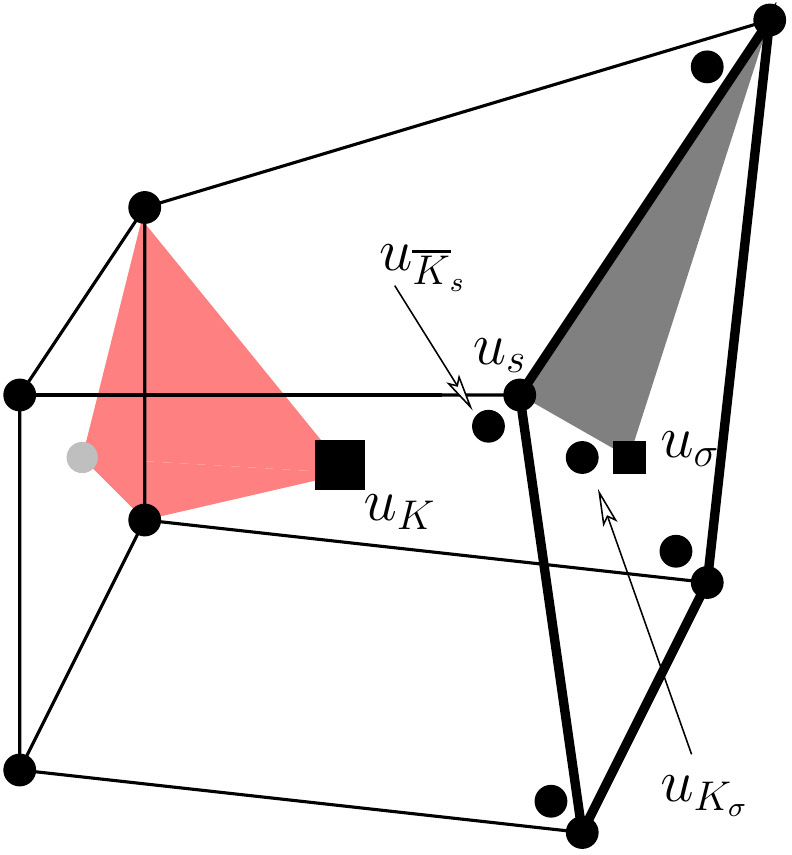}
  \end{minipage}
\end{figure}

We define the VAG-FE scheme's reconstruction operators by
\begin{align}
\begin{aligned}
\label{PiVagFE}
&\bullet\ \Pi_{\D_m} = \Pi_\T,\\
&\bullet\ \Pi_{\D_f} =\wt\Pi_{\D_f} = \Pi_\Delta,\\
&\bullet\ \Pi^\alpha_{\D_m} = \gamma_\alpha\Pi_{\T}\quad\text{ for all }\alpha\in\chi.
\end{aligned}
\end{align}

For the family of \textit{VAG-CV} schemes, reconstruction operators are piecewise constant.
We introduce, for any given $K\in\cells$, a partition
$$
\ov K = \ov\omega_K\cup\(\bigcup_{s\in\nodes_K\setminus\nodes_{ext}}\ov\omega_{K,\ov K_s}\)\cup\(\bigcup_{\sigma\in\faces_K\cap\faces_\G}
 \ov\omega_{K_\sigma}\) .
$$
Similarly, we define for any given $\sigma\in\faces_\G$ a partition
$$
\ov\sigma = \ov\omega_\sigma\cup\(\bigcup_{s\in\nodes_\sigma\setminus\nodes_{ext}}\ov\omega_{\sigma,s}\) .
$$
With each $s\in\nodes\setminus\nodes_{ext}$ and $\ov K_s\in\ov\cells_s$ we associate an open set 
$\omega_{\ov K_s}$, satisfying
$$
\ov\omega_{\ov K_s} = \bigcup_{K\in\ov K_s}\ov\omega_{K,\ov K_s}.
$$
Similarly, for all $s\in\nodes_\G\setminus\nodes_{ext}$ we define $\omega_s$ by
$$
\ov\omega_{s} = \bigcup_{\sigma\in\faces_s\cap\faces_\G}\ov\omega_{\sigma,s}.
$$
We obtain the partitions 
$$
\ov\Omega = \(\bigcup_{\nu\in\dofm\setminus\dofDirm}\ov\omega_\nu\),\quad \ov \Gamma = \(\bigcup_{\nu\in\doff\setminus\dofDirf}\ov\omega_\nu\).
$$
 
We also introduce for each $T = T_{\sigma,s,s'}\in\Delta$ a partition $\ov T = \bigcup_{i=1}^3 \ov T_i$, which we need for the definition of the VAG-CV matrix-fracture interaction operators. We assume that holds 
 $|T_1| = |T_2| = |T_3| = \frac{1}{3} |T|$ in order to preserve the first order convergence of the scheme.\\
 
Finally, we need a mapping between the degrees of freedom of the matrix domain, which are situated on one side of the fracture network, and the set of indices $\chi$. For $K_\sigma\in\dofm$ we have the one-element set 
$\chi(K_\sigma) = \{\alpha\in\chi\mid \n_{K,\sigma} = \n_\alpha\text{ on }\sigma\}$ and therefore the notation
$\alpha(K_\sigma) = \alpha\in\chi(K_\sigma)$.\\

The VAG-CV scheme's reconstruction operators are
\begin{align}
\begin{aligned}
\raisetag{3cm}\label{PiVagCV}
&\bullet\ \Pi_{\D_m} u_{\D_m} =\sum\limits_{\nu\in\dofm\setminus\dofDirm}u_\nu\mathbb 1_{\omega_\nu} ,\\
&\bullet\ \Pi_{\D_f} u_{\D_f}= \sum\limits_{\nu\in\doff\setminus\dofDirf}u_\nu\mathbb 1_{\omega_\nu},\\
&\bullet\ \wt\Pi_{\D_f} u_{\D_f} = \sum\limits_{T_{\sigma,s,s'}\in\Delta} 
( u_{\sigma}\mathbb 1_{T_1} + u_{s}\mathbb 1_{T_2} + u_{s'}\mathbb 1_{T_3}),\\
&\bullet\ \Pi^\alpha_{\D_m}u_{\D_m} = \sum\limits_{T_{\sigma,s,s'}\in\Delta} \sum\limits_{K\in\cells_\sigma}
( u_{K_\sigma}\mathbb 1_{T_1} + u_{\ov K_s}\mathbb 1_{T_2} + u_{\ov K_{s'}}
\mathbb 1_{T_3})\delta_{\alpha(K_\sigma)\alpha}\mathbb 1_{\G_\alpha}.
\end{aligned}
\end{align}

\begin{remark}
\label{remarkdefvag}
The VAG-CV scheme leads us to recover fluxes for the matrix-fracture interactions 
involving degrees of freedom located at the same physical point 
(see subsection \ref{sec_FV}).
\end{remark}

\begin{proposition}
\label{propcclvagfe}
Let us consider a sequence of meshes $(\cells^l)_{l\in\N}$ and let us assume that the sequence $(\T^l)_{l\in\N}$ of tetrahedral submeshes is shape regular, i.e. $\theta_{\T^l}$ is uniformly bounded. We also assume that $\lim_{l \rightarrow \infty} {h_{\T^l}} = 0.$ Then, the corresponding sequence of gradient discretizations $(\D^l)_{l\in\N}$, defined by \eqref{XVag}, \eqref{gradVag}, \eqref{PiVagFE}, is coercive, 
consistent and limit conforming.
\end{proposition}

\begin{proof}
The \textit{VAG-FE} scheme's reconstruction operators are conforming, i.e. $V_\D\subset V^0$. Therefore we deduce coercivity from Proposition \ref{proppoincarecont}. Furthermore we have by partial integration ${\cal W}_\D(\q_m,\q_f) = 0$ for all $(\q_m,\q_f)\in W$. Hence $(\D^l)_{l\in\N}$ is limit conforming.
\par To prove consistency, we need the following prerequisites. 
We define the linear mapping $P_{\D_m}\colon C_\Omega^\infty\to X_{\D_m}^0$ 
such that for all $\psi_m\in C_\O^\infty$ and any cell $K\in \cells$ one has  
\begin{align*}
\begin{array}{l@{\,\,}l@{\,\,}rl}
&(P_{\D_m}\psi_m)_K = \psi_m(\x_K),\\
&(P_{\D_m}\psi_m)_{\ov K_s} = \psi_m(\x_s)  \qquad &\forall s\in\nodes_K,\\
& (P_{\D_m}\psi_m)_{K_\sigma} = \psi_m(\x_\sigma) \qquad &\forall \sigma\in\faces_K\cap\faces_\G. 
\end{array}
\end{align*} 
Likewise, we define the linear mapping $P_{\D_f}\colon C_\G^\infty\to X_{\D_f}^0$ 
such that for all $\psi_f\in C_\G^\infty$ holds
$(P_{\D_f}\psi_f)_\nu = \psi_f(\x_\nu)$ for all $\nu\in dof_{\D_f}$. 
It follows from the classical Finite Element approximation theory and from 
the fact that the interpolation  
$\sum\limits_{s\in\nodes_\sigma}\beta_{\sigma,s} (P_{\D_m}\psi_m)_{\ov K_s}$ 
at the point $\x_\sigma$, $\sigma\in \faces_K\setminus\faces_\Gamma$ is exact 
on cellwise affine functions, that for all $(\psi_m,\psi_f)
\in C_\Omega^\infty\times C_\Gamma^\infty$ holds 
\begin{align}
\label{proofcclvagfe1}
\|\Pi_\T P_{\D_m}\psi_m - \psi_m\|_{H^1(\O\backslash\ov\G)} + \|\Pi_\Delta P_{\D_f}\psi_f - \psi_f\|_{H^1(\G)} 
\leq C(\psi_m,\psi_f,\theta_\T) h_\T. 
\end{align}
The trace inequality implies that for all $v\in H_{\del\O}^1(\O\backslash\ov\G)$ holds 
$$
\|\gamma_{\alpha} v\|_{L^2(\G_\alpha)}\leq C(\O\backslash\ov\G)\|v\|_{H^1(\O\backslash\ov\G)}\quad\text{for }\alpha\in\chi.
$$
We can then calculate for $(u_m,u_f)\in C_\Omega^\infty\times C_\Gamma^\infty$:
\begin{align*}
\mathcal S_\D(u_m,u_f)&\leq 
\sqrt{2}\|\Pi_\T P_{\D_m} u_m - u_m\|_{H^1(\O\backslash\ov\G)} 
+ \sum_{\alpha\in\chi}\|\gamma_\alpha (\Pi_\T P_{\D_m} u_m - u_m)\|_{L^2(\G_\alpha)}\\
 &+ \sum_{i\in I}\sqrt{8}\|\Pi_\Delta P_{\D_f} u_f  - u_f\|_{H^1(\G_i)}\\
 &\leq C(\O\backslash\ov\G,\#\chi,\# I,(u_m,u_f),\theta_\T) ~ h_\T.
\end{align*}
Since $C_\Omega^\infty\times C_\Gamma^\infty$ is dense in $V^0$, the sequence of \textit{VAG-FE} discretisations $(\D_m^l)_{l\in\N}$ is consistent if $h_{\T^l}\to 0$ and $\theta_{\T^l}$ is bounded for $l\to\infty$.
\hfill\qed
\end{proof}

\begin{proposition}
\label{propcclvagcv}
Let us consider a sequence of meshes $(\cells^l)_{l\in\N}$ and let us assume that the sequence $(\T^l)_{l\in\N}$ of tetrahedral submeshes is shape regular, i.e. $\theta_{\T^l}$ is uniformly bounded. We also assume that $\lim_{l \rightarrow \infty} {h_{\T^l}} = 0.$ Then, any corresponding sequence of gradient discretizations $(\D^l)_{l\in\N}$, defined by \eqref{XVag}, \eqref{gradVag}, \eqref{PiVagCV}, is coercive, consistent and limit conforming.
\end{proposition}

\begin{proof}
We combine Lemma \ref{lemmakonstantin} and Proposition \ref{propcclvagfe}. Thus, we have to show that the assumptions of Lemma 
\ref{lemmakonstantin} are satisfied, where $(\overline\D^l)_{l\in\N}$ corresponds to 
the sequence of \textit{VAG-CV} gradient discretisations 
and $(\D^l)_{l\in\N}$ to the corresponding sequence of \textit{VAG-FE} gradient discretisations. 
\par For the following, we define $\faces^\alpha = \bigcup_{i\in I_\alpha}\faces_{\G_i}$ and $\nodes^\alpha 
= \bigcup_{\sigma\in\faces^\alpha}\nodes_\sigma$.
To ease the notation in the proof, we will use, for $\alpha\in\chi$, the uniquely identified mapping $\mu^\alpha\colon\nodes^\alpha\cup\faces^\alpha\subset\doff\to\dofm$, 
defined by $\mu^\alpha(\sigma) = K_\sigma$ (such that $\chi(K_\sigma) = \{\alpha\}$) 
and $\mu^\alpha(s) = \ov K_s$ (for a cell $K$ such that $K\in \cells_\sigma$ with $\sigma\in \faces^\alpha\cap\faces_s$ 
and $\chi(K_\sigma) = \{\alpha\}$).
Let now $\alpha\in\chi$ be fixed. Since the mesh is conforming with respect to the fracture network, 
there is for every $\sigma\in\faces^\alpha$, $e=ss'\in \edges_\sigma$ a $\nu(\sigma,e)\in\{\sigma,s,s'\}$, 
such that
$$
\sup_{\x\in T_{\sigma,e}} |(\Pi_{\D_m}^\alpha v_{\D_m} - \wt\Pi_{\D_f} v_{\D_f})(\x) | = |(\Pi_{\D_m}^\alpha v_{\D_m} - \wt\Pi_{\D_f} v_{\D_f})(\x_{\nu(\sigma,e)}) |
= |v_{\mu^\alpha(\nu(\sigma,e))} - v_{\nu(\sigma,e)} |.
$$
Then we have
\begin{align*}
\| \Pi_{\D_m}^\alpha v_{\D_m} - \wt \Pi_{\D_f} v_{\D_f} \|^2_{L^2(\G_\alpha)}
&\leq \sum_{\sigma\in\faces^\alpha}\sum_{e\in\edges_\sigma}
|T_{\sigma,e} |  |v_{\mu^\alpha(\nu(\sigma,e))} - v_{\nu(\sigma,e)} |^2\\
&\leq 3 
\| \ov\Pi_{\D_m}^\alpha v_{\D_m} - \wt{\ov\Pi}_{\D_f} v_{\D_f} \|^2_{L^2(\G_\alpha)}.
\end{align*}
We have to check \eqref{eqlemmakonstantin0} now. It can be verified that \cite{BM2013}, Lemma 3.4 applies to our case, both, in the matrix domain, where face unknowns might occur, as well as in the fracture network, a domain of codimension 1. This means that we can state that there exist constants $C_m(\theta_\T), C_f(\theta_\T) > 0$, such that
\begin{align}
\|\Pi_{\D_m} u_{\D_m} - \ov\Pi_{\D_m} v_{\D_m}\|_{L^2(\O)} &\leq C_m\cdot h_\T\cdot \|\grad_{\D_m}v_{\D_m}\|_{{L^2(\O)}^d}\quad\text{and} \label{proofcclvagcv1}\\
\|\wt \Pi_{\D_f} v_{\D_f} - \wt {\ov\Pi}_{\D_f} v_{\D_f}\|_{L^2(\G)} = \|\Pi_{\D_f} v_{\D_f} - \ov\Pi_{\D_f} v_{\D_f}\|_{L^2(\G)} &\leq C_f\cdot h_\Delta\cdot \|\grad_{\D_f}v_{\D_f}\|_{{L^2(\G)}^{d-1}} \label{proofcclvagcv2}
\end{align}
For the following calculation we take into account \cite{BM2013}, Lemmata 3.2 and 3.4. We also use that the mesh is conforming with respect to the fracture network and that for $\sigma\in\faces$ and $K\in\cells_\sigma$ (or equivalently for $K\in\cells,\  \sigma\in\faces_K$) holds: $h_K$ is asymptotically equivalent to $h_\sigma$ and $|K|$ is asymptotically equivalent to $h_\sigma |\sigma|$, where $h_K := \max_{\T\ni D\subset K}h_D$ and $h_\sigma := \max_{\Delta\ni T\subset\sigma}h_\T$. Let $\alpha\in\chi,\ \sigma\in\faces^\alpha$ and $K\in\cells_\sigma$, such that $\chi(K_\sigma) = \{\alpha\}$.
Then we have
\begin{align*}
\|\Pi_{\D_m}^\alpha v_{\D_m} &-\ov\Pi_{\D_m}^\alpha v_{\D_m}\|_{L^2(\sigma)}^2
= \|\sum_{\nu\in\{\sigma\}\cup(\nodes_\sigma)} v_{\mu^\alpha(\nu)}(\Pi_{\D_f} \mathfrak e_\nu - \ov\Pi_{\D_f}\mathfrak e_\nu)\|_{L^2(\sigma)}^2  \\
&\leq C \cdot |\sigma|\sum_{s\in\nodes_\sigma}(v_{\mu^\alpha(s)} - v_{\mu^\alpha(\sigma)})^2\\
&\leq C \cdot \frac{|K|}{|h_K|}\cdot \(\sum_{s\in\nodes_K}(v_{\ov K_s} - v_{K})^2 + \sum_{\sigma\in\faces_K\cap\faces_\G}(v_{K_\sigma} - v_K)^2\)
\leq C \cdot h_\sigma\cdot \|\grad_{\D_m}v_{\D_m}\|_{{L^2(K)}^d}^2.
\end{align*}
Therefore
\begin{align}
\label{proofcclvagcv_0}
\|\Pi_{\D_m}^\alpha v_{\D_m} -\ov\Pi_{\D_m}^\alpha v_{\D_m}\|^2_{L^2(\G_\alpha)} 
\leq \sum_{\sigma\in\faces^\alpha}\|\Pi_{\D_m}^\alpha v_{\D_m} -\ov\Pi_{\D_m}^\alpha v_{\D_m}\|^2_{L^2(\sigma)}
\leq C\cdot h_\Delta \cdot\|\grad_{\D_m}v_{\D_m}\|^2_{{L^2(\O)}^d}.
\end{align}
Altogether we obtain
\begin{align*}
\|\Pi_{\D_m} v_{\D_m} &- \ov\Pi_{\D_m} v_{\D_m}\|_{L^2(\O)} 
+ \|\Pi_{\D_f} v_{\D_f} - \ov\Pi_{\D_f} v_{\D_f}\|_{L^2(\G)}
+  \|\wt \Pi_{\D_f} v_{\D_f} - \wt {\ov\Pi}_{\D_f} v_{\D_f}\|_{L^2(\G)} \\
&+ \sum_{\alpha\in\chi}\|\Pi_{\D_m}^\alpha v_{\D_m} -\ov\Pi_{\D_m}^\alpha v_{\D_m}\|_{L^2(\G_\alpha)}
\leq C \cdot (h_\T + h_\Delta + h_\Delta^{\frac{1}{2}})\cdot\| (v_{\D_m},v_{\D_f})\|_\D,
\end{align*}
with a constant $C$ depending only on $\#\chi$ and $\theta_\T$. This proves that \eqref{eqlemmakonstantin0} is satisfied.
\hfill\qed
\end{proof}

\begin{corollary}
The precedent proof shows that $\mathcal{S}_\D(u_m,u_f) = \mathcal{O}(h_\T^{1\over 2})$ for $(u_m,u_f)\in C_\O^\infty\times C_\G^\infty$ 
and that $\mathcal{W}_\D(\q_m,\q_f) = \mathcal{O}(h_\T^{\frac{1}{2}})$ for $(\q_m,\q_f)\in C_{W_m}^\infty\times C_{W_f}^\infty$. 
However, we can prove a higher order of convergence, i.e. 
$\mathcal{W}_\D(\q_m,\q_f) = \mathcal{O}(h_\T)$ for $(\q_m,\q_f)\in C_{W_m}^\infty\times C_{W_f}^\infty$ and 
$\mathcal{S}_\D(u_m,u_f) = \mathcal{O}(h_\T)$ for $(u_m,u_f)\in C_\O^\infty\times C_\G^\infty$.  

\end{corollary}

\begin{proof}
\noindent\textit{Consistency: }
Classically, for all $(\varphi_m,\varphi_f)\in C_\O^\infty\times C_\G^\infty$, 
we have the estimate 
\begin{align*}
&\|\Pi_{\D_m} P_{\D_m}\varphi_m - \varphi_m\|_{L^2(\O)}
+ \|\Pi_{\D_m}^\alpha P_{\D_m}\varphi_m - \trace_\alpha\varphi_m\|_{L^2(\G_\alpha)} \\
& + \|\Pi_{\D_f} P_{\D_f}\varphi_f - \varphi_f\|_{L^2(\G)}+ \|\wt \Pi_{\D_f} P_{\D_f}\varphi_f - \varphi_f\|_{L^2(\G)}
\leq cst(\varphi_m,\varphi_f)\cdot h_\T,
\end{align*}
while \eqref{proofcclvagfe1} grants that holds
$$
\|\nabla_{\D_m}P_{\D_m}\varphi_m - \nabla\varphi\|_{L^2(\O)}+\|\nabla_{\D_f}P_{\D_f}\varphi_f - \nabla\varphi\|_{L^2(\Gamma)} 
\leq cst(\varphi_m,\varphi_f,\theta_{\T}) h_{\T}. 
$$
Taking into account that $C^\infty_{\O}\times C^\infty_{\G}$ is dense in $V$, 
we see that the treated discretisation is consistent with 
$\mathcal{S}_\D(\varphi_m,\varphi_f) = \mathcal{O}(h_\T)$ for $(\varphi_m,\varphi_f)\in C_\O^\infty\times C_\G^\infty$.\\

\noindent\textit{Limit Conformity: }
For all $T \in\Delta$ and 
for all $u_{\D_m}\in X_{\D_m}$ we have that 
$$
\int_{T}(\Pi_{\D_m}^\alpha u_{\D_m} - \ov\Pi_{\D_m}^\alpha u_{\D_m}) \d\tau(\x) = 0.
$$
Introducing the linear operator $P: L^2(\G_\alpha) \rightarrow L^2(\G_\alpha)$ such that 
$P(\varphi) = \frac{1}{|T|}\int_T\varphi \d\tau(\x)$ 
on $T$ for all $T\in\Delta$, we first calculate for any $\q_m\in C^\infty_{W_m}$
\begin{align*}
\|\trace_{\n,\alpha}\q_m - P(\trace_{\n,\alpha}\q_m)\|^2_{L^2(\Gamma_\alpha)} = 
\sum_{\sigma\in \faces_\alpha}\sum_{\Delta\ni T\subset\sigma}\|\trace_{\n,\alpha}\q_m - P(\trace_{\n,\alpha}\q_m)\|^2_{L^2(T)} 
\leq  C(\q_m,\theta_\T) \cdot h^2_\T.
\end{align*}
We proceed:
\begin{align*}
|\int_{\G_\alpha}\trace_{\n,\alpha}\q_m (\Pi_{\D_m}^\alpha u_{\D_m} &- \ov\Pi_{\D_m}^\alpha u_{\D_m}) \d\tau(\x)| \\
& = |\int_{\G_\alpha}(\trace_{\n,\alpha}\q_m - P(\trace_{\n,\alpha}\q_m)) (\Pi_{\D_m}^\alpha u_{\D_m} - \ov\Pi_{\D_m}^\alpha u_{\D_m}) \d\tau(\x)|\\
&\leq\|\trace_{\n,\alpha}\q_m - P(\trace_{\n,\alpha}\q_m)\|_{L^2(\Gamma_\alpha)} \|\Pi_{\D_m}^\alpha u_{\D_m} - \ov\Pi_{\D_m}^\alpha u_{\D_m}\|_{L^2(\Gamma_\alpha)}\\
&\leq C(\q_m,\theta_\T) h_\T^{\frac{3}{2}}\|\grad_{\D_m}u_{\D_m}\|_{L^2(\O)}
\end{align*}
for all $\q_m\in C^\infty_{W_m}$, where we have used \eqref{proofcclvagcv_0} in the last inequality.
We can now conclude by calculating for all for $\q = (\q_m,\q_f)\in C_{W_m}^\infty\times C_{W_f}^\infty$
\begin{align*}
w_{\ov\D}(u_\D,\q) &= (w_{\ov\D} - w_\D)(u_\D,\q) \\
&= \int_\O \div\q_m (\ov\Pi_{\D_m} - \Pi_{\D_m}) u_{\D_m} \d\x 
+\int_{\G} \div_\tau\q_f (\ov\Pi_{\D_f}- \Pi_{\D_f}) u_{\D_f} \d\tau(\x) \\
&+\sum_{\alpha\in\chi}\int_{\G_\alpha} \gamma_{\n,\alpha} \q_m \(
(\wt{\ov\Pi}_{\D_f}  - \wt \Pi_{\D_f}) u_{\D_f} - (\ov\Pi_{\D_f} - \Pi_{\D_f}) u_{\D_f} 
- (\ov\Pi_{\D_m}^\alpha - \Pi_{\D_m}^\alpha) u_{\D_m} 
\) \d\tau(\x)\\
&\leq\|\Pi_{\D_m} u_{\D_m} - \ov\Pi_{\D_m} u_{\D_m}\|_{L^2(\O)} 
\cdot\|\div\q_m\|_{L^2(\O)}\\
&+ \|\Pi_{\D_f} u_{\D_f} - \ov\Pi_{\D_f} u_{\D_f}\|_{L^2(\G)} \cdot\|\div_\tau\q_f\|_{L^2(\G)}
 +  \sum_{\alpha\in\chi}\( (\|\wt \Pi_{\D_f} u_{\D_f} - \wt {\ov\Pi}_{\D_f} u_{\D_f}\|_{L^2(\G_\alpha)}  \\
&+ \|\wt \Pi_{\D_f} u_{\D_f} - \wt {\ov\Pi}_{\D_f} u_{\D_f}\|_{L^2(\G_\alpha)} )
\cdot\|\trace_{\n,\alpha}\q_m\|_{L^2(\G_\alpha)} \\
&+ \int_{\G_\alpha}\trace_{\n,\alpha}\q_m (\Pi_{\D_m}^\alpha u_{\D_m} - \ov\Pi_{\D_m}^\alpha u_{\D_m}) \d\tau(\x)\)
\leq C(\theta_\T,\q)\cdot h_\T\cdot\|u_\D\|_\D,
\end{align*}
where we have taken into account the conformity of $\D$ in the first equation and \eqref{proofcclvagcv1}, \eqref{proofcclvagcv2} in the last inequality.
\hfill\qed
\end{proof}

\begin{remark}
The proofs of Propositions \ref{propcclvagfe} and \ref{propcclvagcv} show that for solutions $(u_m,u_f)\in V^0$ and $(\q_m,\q_f)\in W$ of \eqref{modeleCont} 
such that $u_m\in C^2(\ov K)$, $u_f\in C^2(\ov \sigma)$, $\q_m \in (C^1(\ov K))^d$, $\q_f \in (C^1(\ov \sigma))^{d-1}$ for all 
$K\in \cells$ and all $\sigma\in \Gamma_f$, 
the VAG schemes are consistent and limit conforming of order 1, 
and therefore convergent of order 1.
\end{remark}

\subsection{Hybrid Finite Volume Discretization}
\label{sec_HFV}
In this subsection, the HFV scheme introduced in \cite{EGH09} is extended to the hybrid 
dimensional Darcy flow model. 
We assume here that the faces are planar and that $\x_\sigma$ is the barycenter of $\sigma$ for all $\sigma\in\faces$.\\

\noindent The set of indices $\dofm\times\doff$ for the unknowns is defined by (cf. figure \ref{figureDofHFV})
\begin{align*}
\dofm & = \cells\cup\(\bigcup_{\sigma\in\faces}\ov\cells_\sigma\)\\
\doff  &= \faces_\G\cup\edges_\G,\\
\dofDirm &= \faces_{ext},\\
\dofDirf &= \edges_\Gamma\cap \edges_{ext},
\end{align*}
where for $\sigma\in\faces$ and $K\in\cells_\sigma$
$$
\ov K_\sigma = \left\{
\begin{array}{l l}
\cells_\sigma&\quad\text{ if }\sigma\in\faces\sm\faces_\G\\
\{K\}&\quad\text{ if }\sigma\in\faces_\G.
\end{array}\right.
$$ 
and $\ov\cells_\sigma = \{\ov K_\sigma\mid K\in\cells_\sigma\}$.
We thus have
\begin{align}
\begin{aligned}
\label{XHFV}
X_{\D_m} &= \Bigl\{u_K\bigm | K\in\cells\Bigr\}\cup\Bigl\{u_{\ov K_\sigma}
\bigm |\sigma\in\faces_\G, \ov K_\sigma\in\ov \cells_\sigma\Bigr\},\\
X_{\D_f} &= \Bigl\{u_\sigma\bigm | \sigma\in\faces_\G\Bigr\}\cup\Bigl\{u_e\bigm |  e\in\edges_\G\Bigr\}.
\end{aligned}
\end{align}

The discrete gradients in the matrix (respectively in the fracture domain)
are defined in each cell (respectively in each face) by the 3D (respectively 2D) 
discrete gradients 
\begin{align}
\label{gradHFV}
\grad_{\D_m}\ (\text{resp. }\grad_{\D_f})\text{ as proposed in \cite{EGH09}, pp. 8-9.}
\end{align}

The function reconstruction operators are piecewise constant 
on a partition of the cells and of the fracture faces. 

\begin{figure}
  \begin{minipage}[c]{0.3\textwidth}
\caption{
Cell $K$ touching a fracture face $\sigma$. Illustration of the polyhedron  
and polygone on which:
\newline Red: $\grad_{\D_m}$ is constant.
\newline Grey: $\grad_{\D_f}$ is constant.
}
\label{figureDofHFV}
  \end{minipage}\hfill
    \begin{minipage}[c]{0.6\textwidth}
    \includegraphics[width=0.6\textwidth]{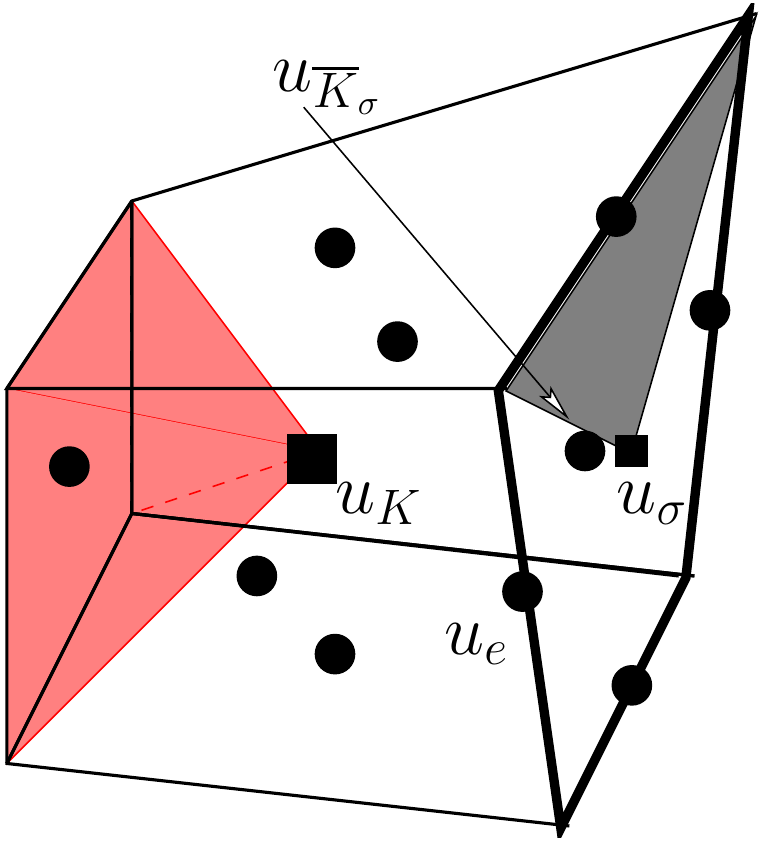}
  \end{minipage}
\end{figure}

\noindent These partitions are respectively denoted, for all $K\in\cells$,  by 
$$
\ov K ~ =  ~\ov\omega_K ~ \cup ~ \( \bigcup_{\sigma\in\faces_K\setminus\faces_{ext}}\ov\omega_{K,\ov K_\sigma} \), 
$$
and, for all $\sigma\in\faces_\G$, by 
$$
\ov\sigma~=~\ov\omega_\sigma~\cup ~ \( \bigcup_{e\in\edges_\sigma\setminus\edges_{ext}}\ov\omega_{\sigma,e} \). 
$$
With each $\sigma\in\faces\setminus\faces_{ext}$ and $\ov K_\sigma\in\ov\cells_\sigma$ we associate an open set 
$\omega_{\ov K_\sigma}$, s.t.
$$
\ov\omega_{\ov K_\sigma} = \bigcup_{K\in\ov K_\sigma}\ov\omega_{K,\ov K_\sigma}.
$$
Similarly, for all $e\in\edges_\G\setminus\edges_{ext}$ we define $\omega_e$ by
$$
\ov\omega_{e} = \bigcup_{\sigma\in\faces_e\cap\faces_\G}\ov\omega_{\sigma,e}.
$$
We obtain the partitions 
$\ov\Omega = \(\bigcup_{\nu\in\dofm\setminus\dofDirm}\ov\omega_\nu\)\quad \ov\Gamma = \(\bigcup_{\nu\in\doff\setminus\dofDirf}\ov\omega_\nu\).
$\\

We also need a mapping between the degrees of freedom of the matrix domain, which are situated on one side of the fracture network, and the set of indices $\chi$. For $\sigma\in\faces_\G$ and $\ov K_\sigma\in\ov\cells_\sigma$ holds by definition $\ov K_\sigma = \{K\}$ for a $K\in\cells_\sigma$ and hence $\n_{\ov K_\sigma} = \n_{K,\sigma}$ is well defined. We obtain the one-element set $\chi(\ov K_\sigma) = \{\alpha\in\chi\mid \n_{\ov K_\sigma} = \n_\alpha\text{ on }\sigma\}$ and therefore the notation $\alpha(\ov K_\sigma) = \alpha\in\chi(\ov K_\sigma)$.

We define the HFV scheme's reconstruction operators by
\begin{align}
\begin{aligned}
\label{PiHFV}
&\bullet\ \Pi_{\D_m} u_{\D_m} =\sum_{\nu\in\dofm\setminus\dofDirm}u_\nu\mathbb 1_{\omega_\nu} ,\\
&\bullet\ \Pi_{\D_f} u_{\D_f}= \sum_{\nu\in\doff\setminus\dofDirf}u_\nu\mathbb 1_{\omega_\nu},\\
&\bullet\ \wt\Pi_{\D_f} u_{\D_f}= \sum_{\sigma\in\faces_\G}u_\sigma\mathbb 1_{\sigma},\\
&\bullet\ \Pi^\alpha_{\D_m}u_{\D_m}  = \sum_{\sigma\in\faces_\G}\sum_{\ov K_\sigma\in\ov\cells_\sigma}\delta_{\alpha(\ov K_\sigma)\alpha}u_{\ov K_\sigma}\mathbb 1_{\sigma}\quad\text{ for all }\alpha\in\chi.
&\end{aligned}
\end{align}

\begin{proposition}
Let us consider a sequence of meshes $(\cells^l)_{l\in\N}$ and let us assume that the sequence $(\T^l)_{l\in\N}$ of tetrahedral submeshes is shape regular, i.e. $\theta_{\T^l}$ is uniformly bounded. We also assume that $\lim_{l \rightarrow \infty} {h_{\T^l}} = 0.$ Then, any corresponding sequence of gradient discretizations $(\D^l)_{l\in\N}$, defined by \eqref{XHFV}, \eqref{gradHFV} and definition \eqref{PiHFV}, is coercive, consistent and limit conforming.
\end{proposition}

\begin{proof} Let us denote in the following by $\Pi_\cells$ and $\Pi_{\faces} = \wt\Pi_{\faces}$ the HFV matrix and fracture reconstruction operators for the special case that $\omega_{\ov K_\sigma} = \emptyset = \omega_e$ for all $\ov K_\sigma\in\bigcup_{\sigma\in\faces}\ov\cells_\sigma$ and $e\in\edges_\G$. 
We start our numerical analysis for HFV by proving the proposition for these special choices and then use Lemma \ref{lemmakonstantin} for generalizing the results.\\

\noindent\textit{Coercivity: }
We first prove that limit conformity against regular test functions, as proved below, implies coercivity.
\par Assume that the sequence of discretizations $(\D^l)_{l\in\N}$ is not coercive. Then we can find a sequence $((u_{\D_m^l},u_{\D_f^l}))_{l\in\N}$ with $(u_{\D_m^l},u_{\D_f^l})\in X_{\D^l}^0$, such that
\begin{equation}
\label{prooflemmacclhfv0}
\|\Pi_{\D_m^l} u_{\D_m^l}\|_{L^2(\O)} + \|\Pi_{\D_f^l} u_{\D_f^l}\|_{L^2(\G)} = 1 \qquad\text{and}\qquad \|(u_{\D_m^l},u_{\D_f^l})\|_{\D^l} < \frac{1}{l}.
\end{equation}
Then follows from a compactness result of \cite{koala} that there exists a $u = (u_m,u_f)\in L^2(\O)\times L^2(\G)$, 
s.t. up to a subsequence
$$
(\Pi_{\D_m^l} u_{\D_m^l},\Pi_{\D_f^l} u_{\D_f^l}) \longrightarrow (u_m,u_f) \qquad\text{ in }L^2(\O)\times L^2(\G)\qquad (\text{ for }l\rightarrow\infty)
$$
and therefore $\|u_m\|_{L^2(\O)} + \|u_f\|_{L^2(\G)} = 1$. On the other hand follows from the discretizations' limit conformity against regular test functions (see below) by Proposition \ref{theoremlimitregularity} and Corollary \ref{cortheoremlimitregularity} that $(u_m,u_f)\in V^0$ and that up to a subsequence
$$
\left\{\begin{array}{r@{\,\,}c@{\,\,}l}
&&\nabla_{\D_m} v_{\D_m^l} \rightharpoonup \nabla v_m \quad\mbox{ in } L^2(\O)^d,\\
&&\nabla_{\D_f} v_{\D_f^l} \rightharpoonup \nabla_\tau v_f \quad\mbox{ in } L^2(\G)^{d-1},\\
&&\wt \Pi_{\D_f} v_{\D_f^l} -\Pi^\alpha_{\D_m} v_{\D_m^l}  \rightharpoonup v_f - \gamma_\alpha v_m\quad\mbox{ in } L^2(\G_\alpha),\text{ for }\alpha\in\chi.\\
\end{array}\right.
$$
Since by construction holds $\|(u_{\D_m^l},u_{\D_f^l})\|_{\D^l}\rightarrow 0$, we obtain $\|(u_m,u_f)\|_{V^0} = 0$. But $\|\cdot\|_{V^0}$ is a norm on $V^0$, which contradicts the fact that $\|u_m\|_{L^2(\O)} + \|u_f\|_{L^2(\G)} = 1$.
\\

\noindent\textit{Consistency: }For $(\varphi_m,\varphi_f)\in C_\O^\infty\times C_\G^\infty$ let us define the projection 
 $P_{\D_m}\varphi_m \in X_{\D_m}^0$ such that for all cell $K\in \cells$ one has  
\begin{align*}
\begin{array}{l@{\,\,}l@{\,\,}rl}
&(P_{\D_m}\varphi_m)_K = \varphi_m(\x_K), & \\
& (P_{\D_m}\varphi_m)_{\ov K_\sigma} = \varphi_m(\x_\sigma) \qquad &\forall \sigma\in\faces_K, 
\end{array}
\end{align*} 
and the projection $P_{\D_f}\varphi_f \in X_{\D_f}^0$ such that 
$(P_{\D_f}\varphi_f)_\nu = \varphi_f(\x_\nu)$ for all $\nu\in dof_{\D_f}$. 
Let us set $v_\D = (P_{\D_m}\varphi_m,P_{\D_f}\varphi_f)$. 
Then holds
$$
\|v_K - \varphi_m\|_{L^2(K)}\leq C_{\varphi_m}\cdot h_\T\cdot |K|^\frac{1}{2}\qquad\text{for }K\in\cells,
$$
where $C_{\varphi_m} := \max_\O \|\grad\varphi_m\|$. Summing over $K\in\cells$ yields
$$
\|\Pi_{\cells}v_{\D_m} - \varphi_m\|_{L^2(\O)}\leq C_{\varphi_m}\cdot h_\T\cdot |\O|^\frac{1}{2} .
$$
We also have
$$
\|v_{\ov K_\sigma} - \trace_\alpha\varphi_m\|_{L^2(\G_\alpha)}\leq c_{\varphi_m}^\alpha\cdot h_\T\cdot |\sigma|^\frac{1}{2}\qquad\text{for }\sigma\in\faces^\alpha,\ \ov K_\sigma\in\ov\cells_\sigma^\alpha
$$
where $c_{\varphi_m}^\alpha := \max_{\G_\alpha} \|\grad_\tau\trace_\alpha\varphi_m\|$, from which we obtain
$$
\|\Pi_{\D_m}^\alpha v_{\D_m} - \trace_\alpha\varphi_m\|_{L^2(\G_\alpha)}\leq c_{\varphi_m}^\alpha\cdot h_\T\cdot |\G_\alpha|^\frac{1}{2}.
$$
Analogously we can derive
$$
\|\Pi_\faces v_{\D_f} - \varphi_f\|_{L^2(\G)}\leq c_{\varphi_f}\cdot h_\T\cdot|\G|^\frac{1}{2},
$$
where $c_{\varphi_f} := \max_\G \|\grad_\tau\varphi_f\|$.
Furthermore, it follows from Lemma 4.3 of \cite{EGH09} that there exists $C>0$ depending only on $\theta_\T$ and $\varphi$ such that
\begin{equation*}
\left.\begin{array}{r@{\,\,}c@{\,\,}ll}
\|\nabla_{\D_m}v_{\D_m} - \nabla\varphi\|_{L^2(\O)}+\|\nabla_{\D_f}v_{\D_f} - \nabla\varphi\|_{L^2(\Gamma)} 
\leq C h_{\T}
\end{array}\right.
\end{equation*}
Taking into account that $C^\infty_{\O}\times C^\infty_{\G}$ is dense in $V^0$, we see that the treated discretisation is consistent.\\

\noindent\textit{Limit Conformity: }Let $\bm{\varphi}_m\in C^\infty_{W_m}$ and for all $K \in \cells,\ \sigma\in \faces_K$ let $\bm{\varphi}_K := \frac{1}{|K|}\int_{K}\bm{\varphi}_m \d\x$ and $\bm{\varphi}_{K,\sigma} := \frac{1}{|\sigma|}\int_{\sigma}\trace_{\n_{K,\sigma} }\bm{\varphi}_m \d\tau(\x)$. In exactly the same manner as \cite{GSDFN}, (29)-(31) are proved, we can show that holds
\begin{align}
A^{12}_{\D_m}u_{\D_m} &\leq C h_\T \|\nabla_{\D_m} u_{\D_m}\|_{{L^2(\Omega)}^d}\qquad\text{and}\label{prooflemmacclhfvcells_1}\\
A^{11}_{\D_m}u_{\D_m} + A^2_{\D_m}u_{\D_m} &- \sum_{\alpha\in\chi}\int_{\Gamma_\alpha}\trace_{\n,\alpha}\bm{\varphi}_{m}(\Pi_{\D_m}^\alpha u_{\D_m}) \d\tau(\x) \notag \\
&=\sum_{K\in \cells}\sum_{\sigma\in \faces_K}|\sigma|(u_K - u_{\ov K_\sigma})(\bm{\varphi}_{K,\sigma} - \bm{\varphi}_{K})\cdot\n_{K,\sigma},\label{prooflemmacclhfvcells_2}
\end{align}
where 
\begin{align*}
A^{11}_{\D_m}u_{\D_m} &:= \sum_{K\in \cells}\sum_{\sigma\in \faces_K}|\sigma|(u_{\ov K_\sigma} - u_K)\boldsymbol{\varphi}_{K}\cdot\n_{K,\sigma},\\
A^{12}_{\D_m}u_{\D_m} &:= \sum_{K\in \cells}\sum_{\sigma\in\faces_K} R_{K,\sigma}(u_{\D_m})\n_{K,\sigma}\cdot \int_{D_{K,\sigma}}\bm{\varphi}_{m} \d\x,\\
A^2_{\D_m} u_{\D_m} &:= \sum_{K\in \cells}\sum_{\sigma\in \faces_K}|\sigma|u_K\bm{\varphi}_{K,\sigma}\cdot\n_{K,\sigma}\qquad\text{and}\\
A^{11}_{\D_m}u_{\D_m} + A^{12}_{\D_m}u_{\D_m} + A^{2}_{\D_m}u_{\D_m} &= \int_{\Omega}  \(\nabla_{\D_m} u_{\D_m} \cdot\bm{\varphi}_m + (\Pi_{\cells} u_{\D_m})\div(\bm{\varphi}_m)\) \d\x,
\end{align*}
with the definition of the gradient stabilization term $R_{K,\sigma}(u_{\D_m})$ as in \cite{EGH09}, pp. 8-9. 
Therefore, applying Cauchy-Schwarz inequality to \eqref{prooflemmacclhfvcells_2}, using the regularity of $\bm{\varphi}_m$, and the estimate \eqref{prooflemmacclhfvcells_1}, we deduce that there exists $C$ depending only on $\bm{\varphi}_m$, $\theta_\T$, such that 
\begin{align*}
\int_{\Omega}  \(\nabla_{\D_m} u_{\D_m} \cdot\bm{\varphi}_m + (\Pi_{\cells} u_{\D_m})\div(\bm{\varphi}_m)\) \d\x
- \sum_{\alpha\in\chi}\int_{\Gamma_\alpha}\trace_{\n,\alpha}\bm{\varphi}_{m}(\Pi_{\D_m}^\alpha u_{\D_m}) \d\tau(\x)
\leq C h_\T \|\nabla_{\D_m} u_{\D_m}\|_{{L^2(\Omega)}^d}.
\end{align*}
Taking into account the result \cite{GSDFN} (33), i.e. for all $\bm\varphi\in C_{W_f}^\infty$ exists a constant $C > 0$ depending only on $\theta_\T$, such that
\begin{align*}
\Big| \int_\G\(\grad_{\D_f}u_{\D_f}\cdot\bm\varphi_f &+ (\Pi_\faces u_{\D_f})\div(\bm\varphi_f)\) \d\tau(\x) \Big| \\
&\leq C h_\Delta\|\grad_{\D_f}u_{\D_f}\|_{{L^2(\G)}^{d-1}},
\end{align*}
we obtain all together
\begin{align*}
w_\D(u_\D,\q) \leq C\cdot h_\T\cdot\|u_\D\|_\D\qquad\text{ for all }\q\in C^\infty_{W_m}\times C^\infty_{W_f}.
\end{align*}
This result is shown above to imply coercivity, which is needed to conclude now.
\par Finally, using that $C^\infty_{W_m}\times C^\infty_{W_f}$ is dense in $W$ and the coercivity of the scheme, 
we derive limit conformity on the whole space of test functions.\\

\noindent\textit{Generalization to arbitrary HFV discretizations: }
We want to apply Lemma \ref{lemmakonstantin}. From \cite{EGH09} Lemma 4.1 and \cite{koala}, it follows that there are positive constants $C_m$ and $C_f$ only depending on $\theta_\T$ and $d$, such that for all $u_\D\in X_\D$ holds
\begin{align*}
\|\Pi_\cells u_{\D_m} - \Pi_{\D_m}u_{\D_m}\|_{L^2(\O)}^2 = \sum_{K\in\cells}\sum_{\sigma\in\faces_K}|\omega_{K,\ov K_\sigma}|(u_K - u_{\ov K_\sigma})^2&\leq C_m\cdot h_\T^2\cdot\|\grad_{\D_m}u_{\D_m}\|_{{L^2(\O)}^d}^2
\\
\|\Pi_\faces u_{\D_f} - \Pi_{\D_f}u_{\D_f}\|_{L^2(\G)}^2 = \sum_{\sigma\in\faces_\G}\sum_{e\in\edges_\sigma}|\omega_{\sigma,e}|(u_\sigma - u_e)^2&\leq C_f\cdot h_\Delta^2\cdot\|\grad_{\D_f}u_{\D_f}\|_{{L^2(\G)}^{d-1}}^2.
\end{align*}
The remaining conditions of Lemma \ref{lemmakonstantin} are trivially satisfied, from what follows the statement of the proposition.
\hfill\qed
\end{proof}

\begin{remark}
The precedent proof shows that for solutions $(u_m,u_f)\in V^0$ and $(\q_m,\q_f)\in W$ of \eqref{modeleCont} 
such that $u_m\in C^2(\ov K)$, $u_f\in C^2(\ov \sigma)$, $\q_m \in (C^1(\ov K))^d$, $\q_f \in (C^1(\ov \sigma))^{d-1}$ for all 
$K\in \cells$ and all $\sigma\in \Gamma_f$, 
the HFV schemes are consistent and limit conforming of order 1, 
and therefore convergent of order 1.
\end{remark}

\subsection{Finite Volume Formulation for \emph{VAG} and \emph{HFV} Schemes}
\label{sec_FV}

For $K\in\cells$ let 
$$
dof_K = \left\{
\begin{array}{l l}
\{\overline K_s, s\in \nodes_K\}\cup \{K_\sigma, \sigma\in \faces_K \cap\faces_\Gamma\}\mbox{ for }  VAG,\\
\{\overline K_\sigma, \sigma\in \faces_K \} \mbox{ for } HFV. 
\end{array}\right.
$$
Analogously, in the fracture domain, for $\sigma\in\faces_\G$ let 
$$
dof_\sigma = \left\{
\begin{array}{l l}
\nodes_\sigma \mbox{ for }  \emph{VAG},\\
\edges_\sigma \mbox{ for } \emph{HFV}. 
\end{array}\right.
$$
Then, for any $\nu\in dof_K$ the discrete \emph{matrix-matrix}-fluxes are defined as 
$$
F_{K\nu}(u_{\D_m})
=  \sum_{\nu'\in dof_K} 
\(\int_K \Lambda_m\grad_{\D_m}\mathfrak{e}_\nu \grad_{\D_m}\mathfrak{e}_{\nu'} \d\x\)  (u_K - u_{\nu'}). 
$$
such that $\int_\Omega \Lambda_m\grad_{\D_m}u_{\D_m} \grad_{\D_m}v_{\D_m} \d\x 
= \sum_{K\in \cells}\sum_{\nu\in dof_K} F_{K\nu}(u_{\D_m})(v_K-v_\nu)$. 
For all $\nu\in dof_\sigma$ the discrete \emph{fracture-fracture}-fluxes are defined as 
$$
F_{\sigma\nu}(u_{\D_f})
= \sum_{\nu'\in dof_\sigma} 
\(\int_\sigma \Lambda_f\grad_{\D_f}\mathfrak{e}_\nu \grad_{\D_f}\mathfrak{e}_{\nu'} \d\tau_f(\x)\) (u_\sigma - u_{\nu'}), 
$$
such that $\int_\Gamma \Lambda_f\grad_{\D_f}u_{\D_f} \grad_{\D_f}v_{\D_f} \d\tau_f(\x) 
= \sum_{\sigma\in \faces_\Gamma}\sum_{\nu\in dof_\sigma} F_{\sigma\nu}(u_{\D_f})(v_\sigma-v_\nu)$. 
To take interactions of the matrix and the fracture domain into account we introduce the set of \emph{matrix-fracture} (\emph{mf}) connectivities 
$$
\C = \{(\nu_m,\nu_f) \,|\, \nu_m\in \dofm^\Gamma, \nu_f\in\doff \mbox{ s.t. } \x_{\nu_m} = \x_{\nu_f}\}
$$ 
with $\dofm^\Gamma = \{\nu\in \dofm \,|\, \x_\nu\in\ov \Gamma\}$. 
The \emph{mf}-fluxes are built such that 
\begin{align*}
& \dsp a_{\D_{mf}}\( (u_{\D_m},u_{\D_f}),(v_{\D_m},v_{\D_f})\) 
= \sum_{(\nu_m,\nu_f)\in \C} F_{\nu_m\nu_f}(u_{\D_m},u_{\D_f})(v_{\nu_m}-v_{\nu_f}) \\
& = \dsp \sum_{i\in I}\int_{\G_i} \frac{T_f}{2\xi-1} 
\!\!\!\!\!\!\!\!
\sum_{\substack{(\alpha,\beta)\in\\ \{(\alpha^\pm(i),\alpha^\mp(i))\}}}
\!\!\!\!\!\!\!\!
\( \xi \Pi_{\D_m}^{\alpha} u_{\D_m} + (1-\xi) \Pi_{\D_m}^{\beta} u_{\D_m} - \wt \Pi_{\D_f} u_{\D_f}\)
\(\Pi_{\D_m}^{\alpha} v_{\D_m} - \wt \Pi_{\D_f} v_{\D_f}\) \d\tau(\x),  
\end{align*}
for all $(v_{\D_m},v_{\D_f}) \in X_\D$.  
For all $\sigma\in \faces_\Gamma$ and $K\in \cells_\sigma$, 
let us denote by $\alpha(K,\sigma)$ the unique $\alpha\in \chi$ such that $\sigma \in \faces_\alpha$ and 
$\n_\alpha = \n_{K,\sigma}$. Let us also set for all $\sigma\in \faces_\Gamma$, 
$(\chi\times \chi)_\sigma = \{(\alpha(K,\sigma),\alpha(L,\sigma)),$ $(\alpha(L,\sigma),\alpha(K,\sigma))\}$ with 
$\cells_\sigma=\{K,L\}$. Then, holds 
\begin{align*}
& a_{\D_{mf}}\( (u_{\D_m},u_{\D_f}),(v_{\D_m},v_{\D_f})\) =  \\
& \sum_{\sigma\in\faces_\Gamma} \sum_{(\alpha,\beta)\in (\chi\times\chi)_\sigma} \int_\sigma {T_f \over 2\xi-1} \(\xi \Pi^\alpha_{\D_m}u_{\D_m}+ (1-\xi)\Pi^\beta_{\D_m}u_{\D_m} - \wt \Pi_{\D_f}u_{\D_f}\)
\(\Pi^\alpha_{\D_m}v_{\D_m} - \wt \Pi_{\D_f}v_{\D_f}\) \d\tau(\x). 
\end{align*}
For all $\sigma\in \faces_\Gamma$, $K\in \cells_\sigma$ and $\x\in \sigma$,  
let us notice that, for the VAG scheme, one has 
$\Pi^{\alpha(K,\sigma)}_{\D_m} \mathfrak{e}_{K_\sigma}(\x) = \wt \Pi_{\D_f}\mathfrak{e}_{\sigma}(\x)$, and 
$\Pi^{\alpha(K,\sigma)}_{\D_m} \mathfrak{e}_{\ov K_s}(\x) = \wt \Pi_{\D_f}\mathfrak{e}_{s}(\x)$
for all $s\in \nodes_\sigma$, and for the HFV scheme, one has 
$\Pi^{\alpha(K,\sigma)}_{\D_m} \mathfrak{e}_{\ov K_\sigma}(\x) = \wt \Pi_{\D_f}\mathfrak{e}_{\sigma}(\x)=1\restriction_\sigma$. 
It result after some computations that the VAG matrix fracture fluxes are defined by 
\begin{align*}
F_{K_\sigma \sigma}(u_{\D_m},u_{\D_f}) = &
\sum_{s\in \nodes_\sigma}
\(\int_\sigma {T_f \over 2\xi-1} (\wt \Pi_{\D_f}\mathfrak{e}_{\sigma})(\wt \Pi_{\D_f}\mathfrak{e}_{s}) \d\tau(\x)\)
\(\xi u_{\ov K_s} + (1-\xi)u_{\ov L_s}- u_s \) \\
&+ \(\int_\sigma {T_f \over 2\xi-1} (\wt \Pi_{\D_f}\mathfrak{e}_{\sigma})^2 \d\tau(\x)\)
\(\xi u_{K_\sigma} + (1-\xi)u_{L_\sigma}- u_\sigma \), 
\end{align*}
for all $\sigma\in \faces_\Gamma$, $\cells_\sigma = \{K,L\}$ , and by 
\begin{align*}
F_{\ov Q_s s}(u_{\D_m},u_{\D_f}) = &\sum_{\sigma\in (\bigcup_{Q\in \ov Q_s} \faces_Q)\cap \faces_s\cap\faces_\Gamma}  
\quad\quad\sum_{K\in \cells_\sigma \cap \ov Q_s,\, L\in\cells_\sigma\setminus\{K\}} \Bigl\{ \\
& 
\(\int_\sigma {T_f \over 2\xi-1} (\wt \Pi_{\D_f}\mathfrak{e}_{s})^2 \d\tau(\x)\)
\(\xi u_{\ov K_{s}} + (1-\xi)u_{\ov L_{s}}- u_{s}\) \\
& + \sum_{s'\in \nodes_\sigma\,|\, ss'\in \edges_\sigma} 
\(\int_\sigma {T_f \over 2\xi-1} (\wt \Pi_{\D_f}\mathfrak{e}_{s'})(\wt \Pi_{\D_f}\mathfrak{e}_{s}) \d\tau(\x)\)
\(\xi u_{\ov K_{s'}} + (1-\xi)u_{\ov L_{s'}}- u_{s'}\) \\
& + 
\(\int_\sigma {T_f \over 2\xi-1} (\wt \Pi_{\D_f}\mathfrak{e}_{\sigma})(\wt \Pi_{\D_f}\mathfrak{e}_{s}) \d\tau(\x)\)
\(\xi u_{K_{\sigma}} + (1-\xi)u_{L_{\sigma}}- u_{\sigma}\) \quad\quad \Bigr\},
\end{align*}
for all $s\in \nodes_\Gamma$, $\ov Q_s\in \ov \cells_s$.  
Similarly the HFV matrix fracture fluxes are defined by 
\begin{align*}
F_{\ov K_\sigma \sigma}(u_{\D_m},u_{\D_f}) = 
 {1 \over 2\xi-1}\(\int_\sigma T_f(\x) \d\tau(\x)\)  \(\xi u_{K_\sigma} + (1-\xi) u_{L_\sigma}- u_\sigma \), 
\end{align*}
for all $\sigma\in \faces_\Gamma$, $\cells_\sigma = \{K,L\}$. 

We observe that for the VAG-CV scheme 
(since 
$\int_\sigma T_f (\wt \Pi_{\D_f}\mathfrak{e}_{s'})(\wt \Pi_{\D_f}\mathfrak{e}_{s}) \d\tau(\x) = 0$ 
for $s\neq s'$ and  $\int_\sigma T_f (\wt \Pi_{\D_f}\mathfrak{e}_{\sigma})(\wt \Pi_{\D_f}\mathfrak{e}_{s}) \d\tau(\x) = 0$)
as well as  for the HFV scheme, the fluxes $F_{\nu_m\nu_f}$ only involves 
the d.o.f. located at the point $\x_{\nu_m}=\x_{\nu_f}$. 

The discrete source terms are defined by
$$
H_\nu = \left\{
\begin{array}{l l}
\displaystyle\int_\Omega h_m\Pi_{\D_m}\mathfrak{e}_\nu \d\x&\quad\text{for }\nu\in\dofm,\\
\displaystyle\int_\G h_f\Pi_{\D_f}\mathfrak{e}_\nu \d\tau_f(\x)&\quad\text{for }\nu\in\doff.
\end{array}\right.
$$

\begin{figure}[h!]
\begin{center}
\includegraphics[height=0.4\textwidth]{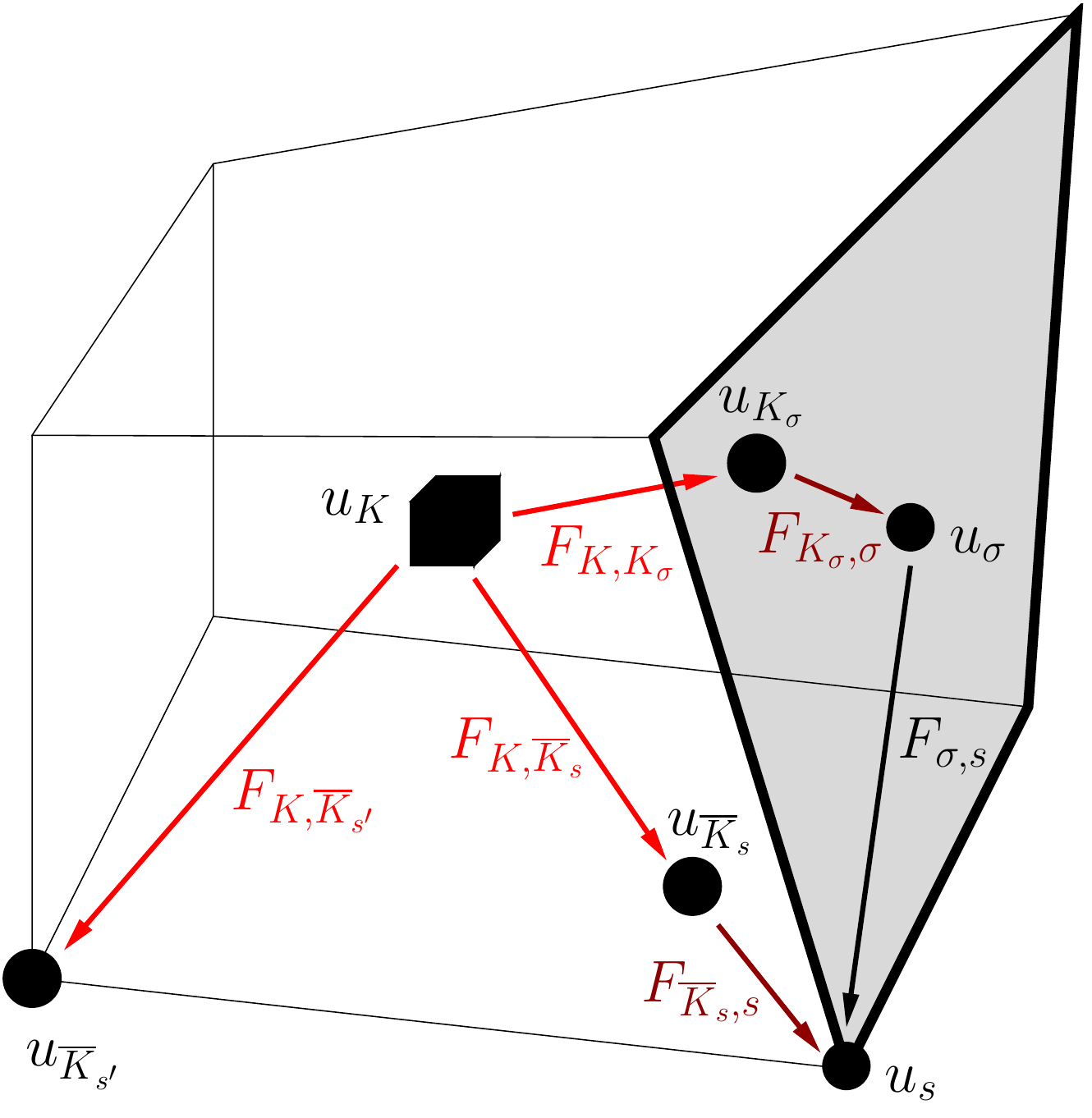}\hspace*{0.1\textwidth}
\includegraphics[height=0.4\textwidth]{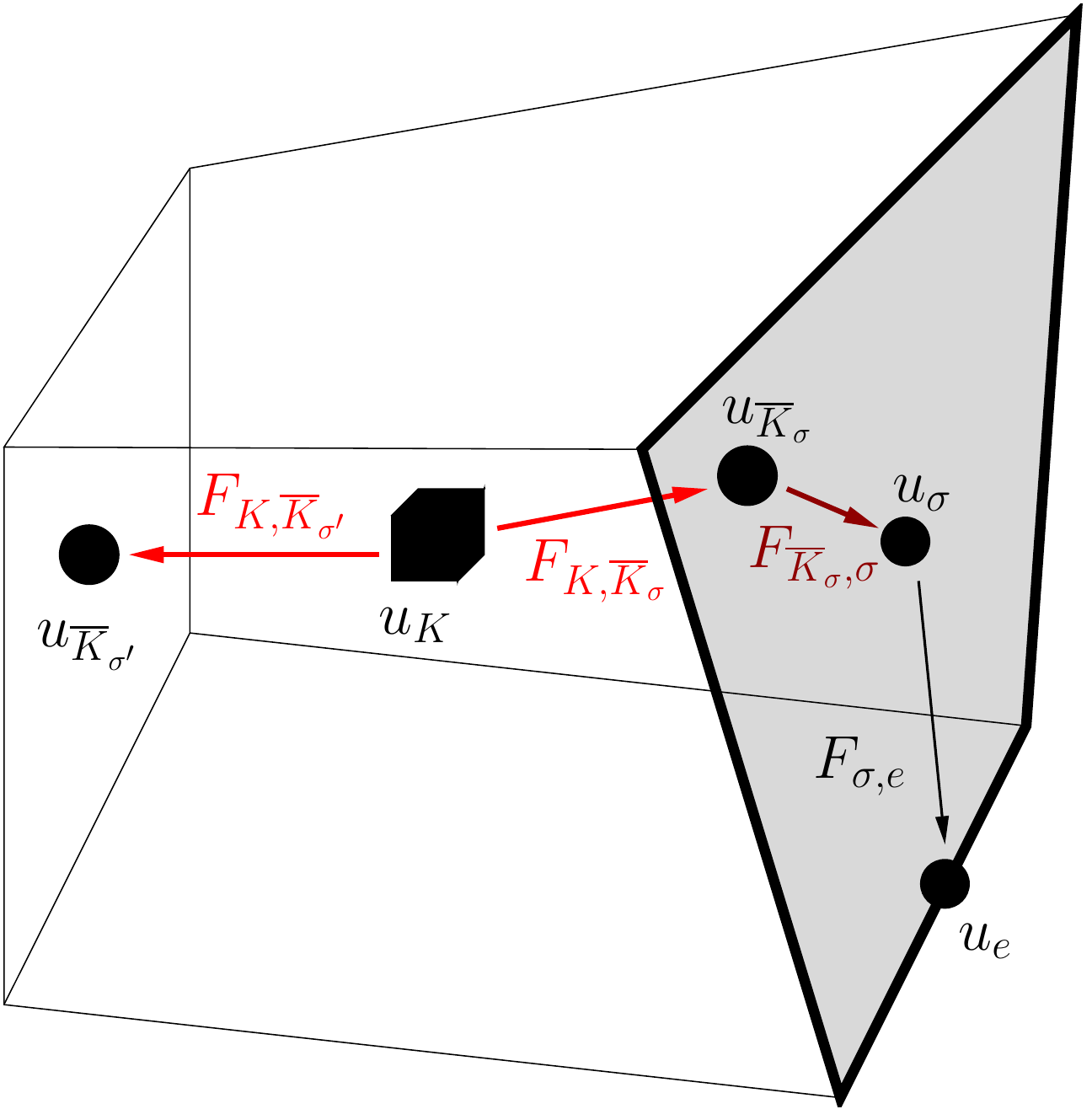}
\caption{\ \emph{mm}-fluxes (red), \emph{mf}-fluxes (dark red) and \emph{ff}-fluxes (black) for VAG (left) and HFV (right) on a 3D cell touching a fracture}
\end{center}
\end{figure}

The following Finite Volume formulation of \eqref{formVar} is equivalent to the discrete variational formulation \eqref{GradientScheme}: 
find $(u_{\D_m},u_{\D_f})\in X_{\D}^0$ such that 
\begin{eqnarray*}
\left\{\begin{array}{l}
\begin{array}{ll}
\text{for all }K\in\cells: & \sum\limits_{\nu\in dof_K}F_{K\nu}(u_{\D_m}) = H_K\\\\
\text{for all }\sigma\in\faces_\G: & \sum\limits_{\nu\in dof_\sigma}F_{\sigma\nu}(u_{\D_f}) 
- \sum\limits_{\substack{\nu_m\in dof_{\D_m}\\\text{s.t. }(\nu_m,\sigma)\in\C}} F_{\nu_m\sigma}(u_{\D_m},u_{\D_f}) = H_\sigma
\end{array}
\\\\
\text{for all }\nu_m\in dof_{\D_m} \sm(\cells\cup \dofDirm):\\\\
\qquad\qquad\quad - \sum\limits_{K\in\cells_{\nu_m}}F_{K\nu_m}(u_{\D_m}) 
+ \sum\limits_{\substack{\nu_f\in dof_{\D_f}\\\text{s.t. }(\nu_m,\nu_f)\in\C}}F_{\nu_m\nu_f}(u_{\D_m},u_{\D_f}) 
= H_{\nu_m}\\\\
\text{for all }\nu_f\in dof_{\D_f} \sm(\faces_\G\cup \dofDirf):\\\\
\qquad\qquad\quad - \sum\limits_{\sigma\in\faces_{\G,\nu_f}}F_{\sigma\nu_f}(u_{\D_f}) 
- \sum\limits_{\substack{\nu_m\in dof_{\D_m}\\\text{s.t. }(\nu_m,\nu_f)\in\C}}F_{\nu_m\nu_f}(u_{\D_m},u_{\D_f}) 
= H_{\nu_f}.
\end{array}\right.
\end{eqnarray*}

Here, $\cells_{\nu_m}$ stands for the set of indices $\{K\in\cells\mid\nu_m\in dof_K\}$ and $\faces_{\G,\nu_f}$ stands for the set  $\{\sigma\in\faces_\G\mid\nu_f\in dof_\sigma\}$.

It is important to note that, using the equation in each cell, the cell unknowns $u_K$, $K\in \cells$, 
can be eliminated without fill-in. 

\section{Numerical Results}
\label{sec_num}
The objective of this numerical section is to compare the VAG-FE, VAG-CV, and the HFV schemes 
in terms of accuracy and CPU efficiency for both Cartesian and tetrahedral meshes on heterogeneous 
isotropic and anisotropic media. For that purpose a family of analytical solutions is built for the fixed 
value of the parameter $\xi=1$. 
We refer to \cite{MJE05}, \cite{ABH09}, \cite{AELH14} 
for a comparison of the solutions 
obtained with different values of the parameter $\xi\in [{1\over 2},1]$ 
with the solution obtained with a 3D representation of the fractures. 

Table \ref{table1} exhibits for the Cartesian  
and tetrahedral meshes, as well as for both the VAG and HFV schemes, 
the number of degrees of freedom (Nb dof), the number of d.o.f. after elimination 
of the cell and Dirichlet unknowns (nb dof el.), and the number of nonzero element 
in the linear system  after elimination without any fill-in of the cell unknowns (Nb Jac). 

In all test cases, the linear system obtained after elimination of the cell unknowns is 
solved using the GMRes iterative solver with the stopping criteria $10^{-10}$. 
The GMRes solver is preconditioned by ILUT \cite{ILUT1}, \cite{ILUT2} 
using the thresholding parameter $10^{-4}$ chosen 
small enough in such a way that all the linear systems can be solved for both schemes and for all meshes. 
In tables \ref{table2} and \ref{table3}, we report the number 
of GMRes iterations $Iter$ and the CPU time taking into account the elimination of the cell unknowns, 
the ILUT factorization, the GMRes iterations, and the computation of the cell values. 

We ran the program on a 2,6 GHz Intel Core i5 processor with 8 GB 1600 MHz DDR3 memory.

\subsection{A class of analytical solutions}

We consider a 3-dimensional open, bounded, simply connected domain $\O=(-0.5,0.5)^3$ with four intersecting fractures $\G_{12} = \{(x,y,z)\in\O\mid x=0,y>0\}$, $\G_{23} = \{(x,y,z)\in\O\mid y=0,x>0\}$, $\G_{34} = \{(x,y,z)\in\O\mid x=0,y<0\}$ and $\G_{14} = \{(x,y,z)\in\O\mid y=0,x<0\}$.
We also introduce the piecewise disjoint, connex subspaces of $\O$, $\O_1 = \{(x,y,z)\in\O\mid y>0,x<0\}$, $\O_2 = \{(x,y,z)\in\O\mid y>0,x>0\}$, $\O_3 = \{(x,y,z)\in\O\mid y<0,x>0\}$ and $\O_4 = \{(x,y,z)\in\O\mid y<0,x<0\}$.

\paragraph{Derivation:}
For $(u_m,u_f)\in V$, we denote
$
u_m(x,y,z) = u_i(x,y,z)\text{ on }\O_i,\, { i = 1,\dots,4}
$
and
$
u_f(x,y,z) = u_{ij}(y,z)\text{ on }\G_{ij},\, ij\in J,
$
where we have introduced $J = \{12,23,34,14\}$. We assume that a solution of the discontinuous pressure model writes in the fracture network
$
u_{ij}(y,z) = \alpha_f(z) + \beta_{ij}(z)\gamma_{ij}(y),\, ij\in J
$
and in the matrix domain
\begin{eqnarray*}
\left\{\begin{array}{r@{\,\,}c@{\,\,}ll}
u_{1}(x,y,z) &=& \alpha_1(z)u_{12}(y,z)u_{14}(x,z)\\
u_{2}(x,y,z) &=& \alpha_2(z) u_{12}(y,z)u_{23}(x,z)\\
u_{3}(x,y,z) &=& \alpha_3(z) u_{34}(y,z)u_{23}(x,z)\\
u_{4}(x,y,z) &=& \alpha_4(z) u_{34}(y,z)u_{14}(x,z).
\end{array}\right.
\end{eqnarray*}
On $\gamma_{ij},ij\in J$ we assume
$
\gamma_{ij}(0) = 0,
$
such that the continuity of $u_f$ is well established at the fracture-fracture intersection, as well as
$
\gamma_{ij}'(0) = 1,
$
to ease the following calculations. 
For $i = 1,\dots,4$ let $K_i = \Lambda_{m}\restriction_{\Omega_i}$ and for $ij\in J$ let $T_{ij} = T_f\restriction_{\G_{ij}}$.
From the conditions 
$
\gamma_{\n,\alpha} \q_m = T_f (\gamma_\alpha u_m - u_f) \mbox{ on } \G_\alpha,\ \alpha\in\chi,
$
we then get, after some effort in computation,
\begin{equation}
\label{alphai}
\left.\begin{array}{lll}
&\dsp\alpha_1(z) = \dsp \( \alpha_f(z) - \frac{K_{1y}}{T_{14}}\beta_{12}(z) \)^{-1},  
&\dsp\alpha_2(z) = \dsp \( \alpha_f(z) - \frac{K_{1y}K_{2x}K_{3y}K_{4x}}{K_{1x}K_{3x}K_{4y}T_{23}}\beta_{12}(z) \)^{-1},\\
&\dsp\alpha_3(z) = \dsp \( \alpha_f(z) - \frac{K_{1y}K_{3y}K_{4x}T_{12}}{K_{1x}K_{4y}T_{23}T_{34}}\beta_{12}(z) \)^{-1},
&\dsp\alpha_4(z) = \dsp \( \alpha_f(z) - \frac{K_{1y}K_{4x}T_{12}}{K_{1x}T_{14}T_{34}}\beta_{12}(z) \)^{-1},\\
&\dsp\beta_{23}(z) = \dsp \frac{K_{1y}K_{3y}K_{4x}T_{12}}{K_{1x}K_{3x}K_{4y}T_{23}}\beta_{12}(z),
&\dsp \beta_{34}(z) = \dsp -\frac{K_{1y}K_{4x}T_{12}}{K_{1x}K_{4y}T_{34}}\beta_{12}(z)\\
&\dsp \beta_{14}(z) = \dsp -\frac{K_{1y}T_{12}}{K_{1x}T_{14}}\beta_{12}(z), 
&\dsp \frac{K_{1y}K_{2x}K_{3y}K_{4x}}{K_{1x}K_{2y}K_{3x}K_{4y}} = 1.
\end{array}\right.
\end{equation}
Obviously, we have taken $\alpha_f$ and $\beta_{12}$ as degrees of freedom, here. 
However, these functions must be chosen in such a way that $\frac{1}{\alpha_i(z)}\neq 0$ for $i=1,\dots,4$.
\begin{remark}
We would like to explicitly calculate the jump at the matrix-fracture interfaces for this class of solutions. At $\G_{ij}$ we have
\begin{align*}
u_i(0,y,z) - u_j(0,y,z) &= ( \alpha_i(z) - \alpha_j(z) )\cdot\alpha_f(z)\cdot u_{ij}(y,z),\qquad \text{ for }ij\in \{12,34\}\\
u_i(x,0,z) - u_j(x,0,z) &= ( \alpha_i(z) - \alpha_j(z) )\cdot\alpha_f(z)\cdot u_{ij}(x,z),\qquad \text{ for }ij\in \{23,14\}.
\end{align*}
From \eqref{alphai}, we observe, that the pressure becomes continuous at the matrix-fracture interfaces, as the $T_{ij}$ tend to $\infty$ uniformly.
\end{remark}

\begin{remark}
In order to obtain solutions with discontinuities at the matrix-fracture interfaces, we had to omit the constraint of flux conservation at fracture-fracture intersections.
\end{remark}

\subsection{Test Case}
We define a solution by setting
$\alpha_f(z) = e^{sin(\pi z)}$, 
$\beta_{12}(z) = -1$, 
$\gamma_{12}(y) = cos(2\pi y) + y - 1$, 
$\gamma_{23}(x) = x$, 
$\gamma_{34}(y) = - e^{cos(\pi y)} + y + e$, 
$\gamma_{14}(x) = \frac{sin(\pi x)}{\pi}$. 
The parameters we used for the different test cases are
\begin{itemize}
\item Isotropic Heterogeneous Permeability:
\begin{align*}
K_{1x} &= K_{1y} = K_{1z} =  1,\ K_{2x} = K_{2y} = K_{2z} = 100,\\
K_{3x} &= K_{3y} = K_{3z} = 3,\ K_{4x} = K_{4y} = K_{4z} = 40,\\
T_{12} &= 1,\ T_{23} = 0.2,\ T_{34} = 100,\ T_{14} = 10,\\
K_{12} &= 1,\ K_{23} = 2,\ K_{34} = 3,\ K_{14} = 10.
\end{align*}
\item Anisotropic Heterogeneous Permeability:
\begin{align*}
K_{1x} &= K_{1z} = 1,\ K_{1y} = 50,\ K_{2x} = K_{2z} = 2,\ K_{2y} = 100,\\ 
K_{3y} &= K_{3z} = 3,\ K_{3x} = 30,\  K_{4z} = 4,\ K_{4x} = K_{4y} = 40,\\ 
T_{12} &= T_{23} = T_{34} = T_{14} = 1,\\
K_{12} &= K_{23} = K_{34} = K_{14} = 1.
\end{align*}
\end{itemize}
In the following figures we plot the normalized $L^2$ norms of the errors, which are calculated as follows:
\begin{itemize}
\item normalized error of the solution: $err_{sol} = \frac{\|\Pi_{\D_m}u_{\D_m} - u_m\|_{L^2(\O)} + \|\Pi_{\D_f}u_{\D_f} - u_f\|_{L^2(\G)}}
{\|u_m\|_{L^2(\O)} + \|u_f\|_{L^2(\G)}}$
\item normalized error of the gradient: $err_{grad} = \frac{ \|\grad_{\D_m}u_{\D_m} - \grad u_m\|_{{L^2(\O)}^d} + \|\grad_{\D_f}u_{\D_f} 
- \grad_\tau u_f\|_{{L^2(\G)}^{d-1}} }
{ \|\grad u_m\|_{{L^2(\O)}^d} + \|\grad_\tau u_f\|_{{L^2(\G)}^{d-1}} }$
\end{itemize}
In the following tables is additionally found the normalized error of the jump: \linebreak
$err_{jump} = \frac{ \sum_{\alpha\in\chi} \|\Pi_{\D_m}^\alpha u_{\D_m} - \Pi_{\D_f}u_{\D_f} - \trace_\alpha u_m + u_f\|_{L^2(\G_\alpha)} }
{ \sum_{\alpha\in\chi} \|\trace_\alpha u_m + u_f\|_{L^2(\G_\alpha)} }$.\\

\begin{table}[H]
\begin{center}
\resizebox{0.47\textheight}{!}{
\begin{tabular}{|c|c|c|c|c|c|c|c|c|c|}
\cline{3-8}
\multicolumn{2}{c|}{}  & \multicolumn{3}{c|}{\textbf{\textit{VAG}}}  & \multicolumn{3}{c|}{\textbf{\textit{HFV}}} \\ \hline
\multicolumn{8}{|l|}{\textit{Hexahedral Meshes}}                                                                                                                                                                                                                                                                                                \\ \hline
\multicolumn{1}{|l|}{\textbf{Key}} & \multicolumn{1}{l|}{\textbf{Nb Cells}} 
& \multicolumn{1}{l|}{\textbf{Nb dof}} & \multicolumn{1}{l|}{\textbf{Nb dof el.}} & \multicolumn{1}{l|}{\textbf{Nb Jac}} 
& \multicolumn{1}{l|}{\textbf{Nb dof}} & \multicolumn{1}{l|}{\textbf{Nb dof el.}} & \multicolumn{1}{l|}{\textbf{Nb Jac}}  \\ \hline
1 &         512 &       1949 &        1437 &       31253 &        2776 &        2264 &       20696    \\\hline
2 &        4096 &       11701 &        7605 &      178845 &		19248 &       15152 &      150320   \\\hline
3 &       32768 &      79205 &       46437 &    1154861 &	142432 &      109664 &     1141856 \\\hline
4 &      262144 &      578245 &      316101 &     8152653 &  1093824 &      831680 &     8892608\\\hline
5 &     2097152 &     4408709 &     2311557 &    60910733 &  	8569216 &     6472064 &    70173056
 \\\hline\hline  
 \multicolumn{8}{|l|}{\textit{Tetrahedral Meshes}}                                                                                                                                                                                                                                                                                                \\ \hline
6 &        1337 &        2514 &        1177 &       18729 &        4943 &        3606 &       22642 \\ \hline
7 &       10706 &       15765 &        5059 &       81741 &        35520 &       24814 &      164246\\ \hline
8 &      100782 &      131204 &       30422 &      492158 &      317367 &      216585 &     1474817 \\ \hline
9 &      220106 &      279281 &       59175 &      956659 &      685718 &      465612 &     3190244 \\ \hline
10 &      428538 &      533442 &      104904 &     1694008 &     1324614 &      896076 &     6167300 \\ \hline
11 &     2027449 &     2452416 &      424967 &     6818299 &     6193783 &     4166334 &    28862986 \\ \hline\hline
\end{tabular}}
\caption{\textbf{Key} defines the mesh reference; \textbf{Nb Cells} is the number of cells of the mesh;
\textbf{Nb dof} is the number of discrete unknowns; \textbf{Nb dof el.} is the number of discrete unknowns after elimination of cell unknowns; \textbf{Nb Jac} refers to the number of non-zero Jacobian entries after elimination of the cell unknowns and equations.}
\label{table1}
\end{center}
\end{table}

\begin{figure}[H]
\begin{center}
\includegraphics[height=0.285\textheight]{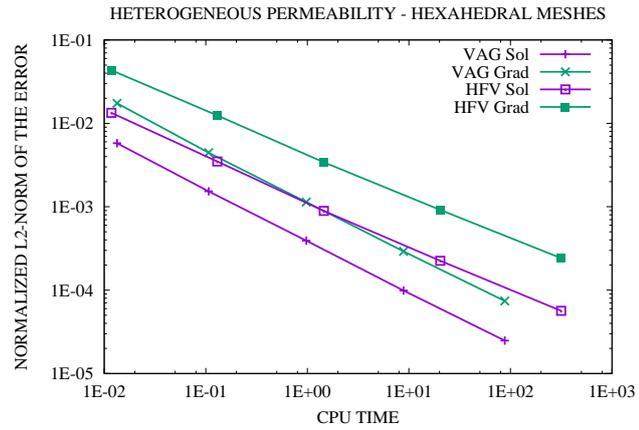}   
\includegraphics[height=0.285\textheight]{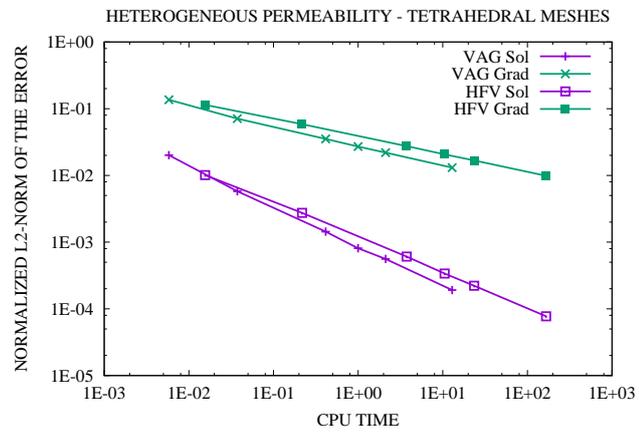}   
\caption{Heterogeneous Permeability: Comparison of VAG-FE and HFV on hexahedral and 
tetrahedral meshes.}
\end{center}
\end{figure}
\begin{figure}[h!]
\begin{center}
\includegraphics[height=0.285\textheight]{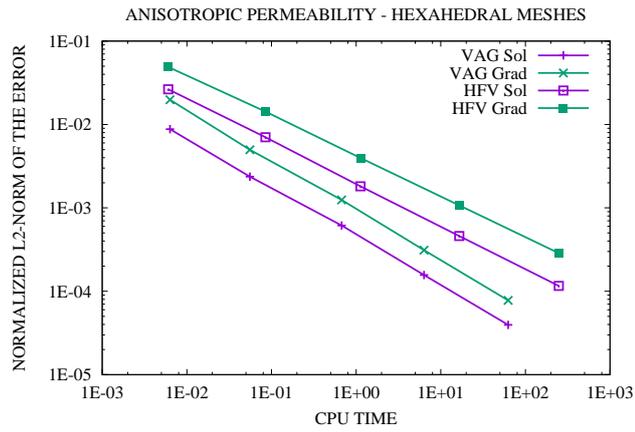}   
\includegraphics[height=0.285\textheight]{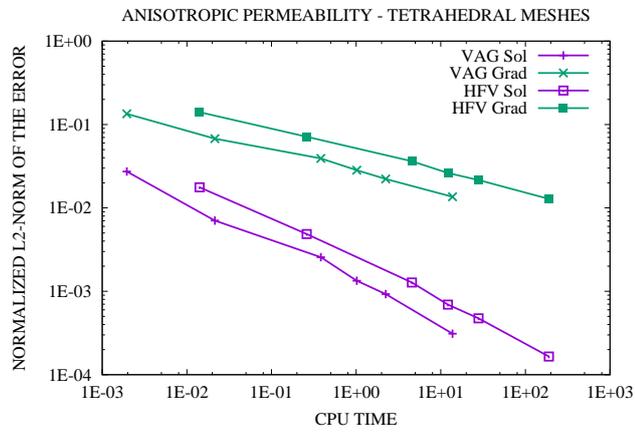}   
\caption{Anisotropic Permeability: Comparison of VAG-FE and HFV on hexahedral and tetrahedral meshes.}
\end{center}
\end{figure}

\clearpage

\begin{table}[H]
\begin{center}
\resizebox{0.47\textheight}{!}{
\begin{tabular}{|c|c|c|c|c|c|c|c|c|}
\cline{2-9}
\multicolumn{1}{c|}{}  & \multicolumn{8}{c|}{\textbf{\textbf{Heterogeneous Permeability: }\textit{VAG}}} 
\\\hline
\multicolumn{9}{|l|}{\textit{Hexahedral Meshes}}
\\ \hline
\multicolumn{1}{|l|}{\textbf{Key}}
& \multicolumn{1}{|l|}{\textbf{Iter}} & \multicolumn{1}{|l|}{\textbf{CPU}} & \multicolumn{1}{|l|}{\textbf{$err_{sol}$}}  
& \multicolumn{1}{|l|}{\textbf{$err_{grad}$}} & \multicolumn{1}{|l|}{\textbf{$err_{jump}$}}
& \multicolumn{1}{|l|}{\textbf{$\alpha_{sol}$}} & \multicolumn{1}{|l|}{\textbf{$\alpha_{grad}$}} & \multicolumn{1}{|l|}{\textbf{$\alpha_{jump}$}}
\\ \hline
1 &	  	  8 &	   1.34E-2 &	   5.78E-3 &	   1.74E-2 &	   8.99E-3 &		   1.92 &	        1.97 &	        1.83	\\\hline
2 &	 	 12 &   0.11 &	       1.53E-3 &	   4.44E-3 &	   2.53E-3 &	   		1.92 &	        1.97 &	        1.83 \\\hline
3 &	      22 &  0.98 &		   3.92E-4 &   1.14E-3 &	   6.72E-4 &		   1.97 &	        1.96	&	        1.91  \\\hline
4 &    	 41 &   8.86 &         9.89E-5 &   2.91E-4 &	   1.73E-4 &		   1.99 &			1.97 &	        1.96 \\\hline
5 &		  79 &   87.91 &      2.48E-5 &   7.40E-5 &	   4.40E-5 &		  	1.99 &        1.98 &	        1.98
 \\\hline\hline
 \multicolumn{9}{|l|}{\textit{Tetrahedral Meshes}}
\\ \hline
6 &		  7 &	   5.82E-3 &	      2.01E-2 &	  0.14 &        2.25E-2 &	   1.80 &       0.94 &        1.68	  	\\\hline
7 &		  10 &   3.73E-2 &	   5.78E-3 &   7.09E-2 &   7.03E-3 &	   1.80 &       0.94 &        1.68		 \\\hline
8 &  		  20 &	  0.41 &	      1.44E-3 &   3.52E-2 &   1.81E-3 &	   1.86 &       0.94 &        1.82   \\\hline
9 &        26 &	  1.00 &	       8.11E-4 &   2.71E-2 &   1.06E-3 &	   2.20 &        1.01 &        2.06      \\\hline
10 &		  32 &	   2.11 &		   5.60E-4 &   2.19E-2 &   7.36E-4 &       1.67 &       0.95 &       1.62    \\\hline
11 &	 	 53 &	   12.92 &		   1.92E-4 &   1.31E-2 &   2.58E-4 &	   2.07 &        1.00 &        2.03
 \\\hline\hline
\end{tabular}}
\vskip 0.5cm
\resizebox{0.47\textheight}{!}{
\begin{tabular}{|c|c|c|c|c|c|c|c|c|}
\cline{2-9}
\multicolumn{1}{c|}{}  & \multicolumn{8}{c|}{\textbf{\textbf{Heterogeneous Permeability: }\textit{HFV}}} 
\\\hline
\multicolumn{9}{|l|}{\textit{Hexahedral Meshes}}
\\ \hline
\multicolumn{1}{|l|}{\textbf{Key}}
& \multicolumn{1}{|l|}{\textbf{Iter}} & \multicolumn{1}{|l|}{\textbf{CPU}} & \multicolumn{1}{|l|}{\textbf{$err_{sol}$}}  
& \multicolumn{1}{|l|}{\textbf{$err_{grad}$}} & \multicolumn{1}{|l|}{\textbf{$err_{jump}$}}
& \multicolumn{1}{|l|}{\textbf{$\alpha_{sol}$}} & \multicolumn{1}{|l|}{\textbf{$\alpha_{grad}$}} & \multicolumn{1}{|l|}{\textbf{$\alpha_{jump}$}}
\\ \hline
1 &	  	   11 &	   1.18E-2 &	      1.34E-2 &	   4.3E-2 &	   2.15E-2 &	   1.94 &        1.80 &        1.98	\\\hline
2 &	 	  19 &		  0.13 &	       3.49E-3 &   1.24E-2 &	   5.44E-3 &	   1.94 &        1.80 &        1.98 \\\hline
3 &	       35 &	   1.45 &			   8.91E-4 &   3.41E-3 &	   1.38E-3 &	   1.97 &        1.86 &        1.98  \\\hline
4 &    	  73 &	   20.36 &		        2.25E-4 &   9.15E-4 &   3.47E-4 &	   1.99  &       1.90 &        1.99 \\\hline
5 &		   141 &   315.38 &        5.65E-5 &	   2.42E-4 &   8.69E-5 &	   1.99 &        1.92 &        2.00
 \\\hline\hline
 \multicolumn{9}{|l|}{\textit{Tetrahedral Meshes}}
\\ \hline
6 &		   12 &		   1.56E-2 &	   1.01E-2 &		  0.11 &        1.74E-2  &	   1.88 &       0.96 &        1.73  	\\\hline
7 &		   21 &		  0.22 &	       2.74E-3 &		   5.87E-2 &   5.24E-3 &	   1.88 &       0.96 &        1.73		 \\\hline
8 &  		   43 &		   3.75 &	       6.07E-4 &		   2.75E-2 &   1.17E-3 &	   2.02 &        1.02 &        2.00   \\\hline
9 &         60 &		   10.51 &		   3.38E-4 &		   2.07E-2 &   6.62E-4 &	   2.25 &        1.08 &        2.20      \\\hline
10 &		  73 &		   23.52 &	       2.22E-4 &		   1.68E-2 &   4.37E-4 &	   1.90 &	       0.94 &        1.87     \\\hline
11 &	 	 119 &		   166.46 &      7.73E-5 &	   9.87E-3 &   1.58E-4 &	   2.03 &        1.02 &        1.96 
 \\\hline\hline
\end{tabular}}
\end{center}
\caption{ Isotropic test case. \textbf{Key} refers to the mesh defined in table \ref{table1}; \textbf{Iter} is the number of solver iterations; 
\textbf{CPU} refers to the solver CPU time in seconds; \textbf{$err_{sol}, err_{grad}, err_{jump}$} are the respective $L^2$-errors as defined above;
\textbf{$\alpha_{sol}, \alpha_{grad}, \alpha_{jump}$} are the orders of convergence of the solution, of the gradient and of the jump, respectively.}
\label{table2}
\end{table}
%
\begin{table}[H]
\begin{center}
\resizebox{0.47\textheight}{!}{
\begin{tabular}{|c|c|c|c|c|c|c|c|c|}
\cline{2-9}
\multicolumn{1}{c|}{}  & \multicolumn{8}{c|}{\textbf{\textbf{Anisotropic Permeability: }\textit{VAG}}} 
\\\hline
\multicolumn{9}{|l|}{\textit{Hexahedral Meshes}}
\\ \hline
\multicolumn{1}{|l|}{\textbf{Key}}
& \multicolumn{1}{|l|}{\textbf{Iter}} & \multicolumn{1}{|l|}{\textbf{CPU}} & \multicolumn{1}{|l|}{\textbf{$err_{sol}$}}  
& \multicolumn{1}{|l|}{\textbf{$err_{grad}$}} & \multicolumn{1}{|l|}{\textbf{$err_{jump}$}}
& \multicolumn{1}{|l|}{\textbf{$\alpha_{sol}$}} & \multicolumn{1}{|l|}{\textbf{$\alpha_{grad}$}} & \multicolumn{1}{|l|}{\textbf{$\alpha_{jump}$}}
\\ \hline
1 &	  	 7 &	   6.32E-3 &		   8.78E-3 &	   1.98E-2 &	   8.69E-3 &    1.89 &        1.99 &        1.89		\\\hline
2 &	 	 9 &	   5.56E-2 &		   2.37E-3 &    4.97E-3 &   2.34E-3 &	   1.89 &        1.99 &        1.89 \\\hline
3 &	     14 &  0.67 &		       6.15E-4 &	   1.24E-3 &    6.06E-4 &   1.95 &        2.00 &        1.95  \\\hline
4 &    	 26 &   6.35 &		       2.28E-4  &   1.57E-4 &   3.11E-4 &    1.97 &        2.00 &		 1.97 \\\hline
5 &		 47 &   62.65 &		    3.95E-5 &   7.78E-5 &	   3.89E-5 &  1.99 &        2.00 & 			1.99
 \\\hline\hline
 \multicolumn{9}{|l|}{\textit{Tetrahedral Meshes}}
\\ \hline
6 &		 7&	   1.95E-3 &	      2.73E-2 &	  0.13 &        2.70E-2 &	   1.95 &       0.99 &        1.95	  	\\\hline
7 &		 8 &	   2.14E-2 &	      7.05E-3 &    6.76E-2 &   6.98E-3 &	   1.95 &       0.99 &        1.95		 \\\hline
8 &  		 15 &   0.38 &		      2.56E-3 &   3.92E-2 &   2.53E-3 &	   1.35 &       0.73 &        1.36   \\\hline
9 &       21 &   1.02 &	          1.34E-3 &   2.84E-2 &   1.32E-3 &	   2.49 &        1.24 &        2.49      \\\hline
10 &		  25 &   2.24 &	          9.26E-4 &   2.22E-2 &   9.14E-4 &	   1.66 &        1.10 &        1.67   \\\hline
11 &	 	 41 &   13.78 &	       3.10E-4 &   1.36E-2 &   3.07E-4 &	   2.11 &       0.95 &        2.11
 \\\hline\hline
\end{tabular}}
\vskip 0.5cm
\resizebox{0.47\textheight}{!}{
\begin{tabular}{|c|c|c|c|c|c|c|c|c|}
\cline{2-9}
\multicolumn{1}{c|}{}  & \multicolumn{8}{c|}{\textbf{\textbf{Anisotropic Permeability: }\textit{HFV}}} 
\\\hline
\multicolumn{9}{|l|}{\textit{Hexahedral Meshes}}
\\ \hline
\multicolumn{1}{|l|}{\textbf{Key}}
& \multicolumn{1}{|l|}{\textbf{Iter}} & \multicolumn{1}{|l|}{\textbf{CPU}} & \multicolumn{1}{|l|}{\textbf{$err_{sol}$}}  
& \multicolumn{1}{|l|}{\textbf{$err_{grad}$}} & \multicolumn{1}{|l|}{\textbf{$err_{jump}$}}
& \multicolumn{1}{|l|}{\textbf{$\alpha_{sol}$}} & \multicolumn{1}{|l|}{\textbf{$\alpha_{grad}$}} & \multicolumn{1}{|l|}{\textbf{$\alpha_{jump}$}}
\\ \hline
1 &	  	  9 &		   6.02E-3 &	      2.64E-2 &	   4.89E-2 &	   3.35E-2 &		1.91 &        1.78 &        2.01	\\\hline
2 &	 	 16 &	   8.48E-2 &		  7.02E-3 &	   1.43E-2 &	   8.30E-3 &	   	1.91 &        1.78 &        2.01 \\\hline
3 &	      29 &	   1.13 &		       1.81E-3 &	   3.96E-3 &	   2.07E-3 &	   1.95 &	        1.85 &	        2.00  \\\hline
4 &    	  55 &	   16.55 &	        4.60E-4 &	   1.07E-3 &	   5.19E-4 &	   1.98 &	        1.89 &        2.00 \\\hline
5 &		  108 &   248.20 &        1.16E-4 &	   2.86E-4 &	   1.30E-4 &	   1.99 &	        1.91 &        2.00
 \\\hline\hline
 \multicolumn{9}{|l|}{\textit{Tetrahedral Meshes}}
\\ \hline
6 &		  10 &	   1.41E-2 &   1.77E-2 &	  0.14 &        1.79E-2 &	   1.86 &       0.98 &        1.91	  	\\\hline
7 &		  19 &	  0.26 &		   4.86E-3 &   7.13E-2 &   4.75E-3 &	   1.86 &       0.98 &        1.91	 \\\hline
8 &  		  37 &	   4.56 &	       1.28E-3 &   3.63E-2 &   1.21E-3 &   	1.79 &       0.90 &        1.83   \\\hline
9 &        47 &	   12.16 &		   6.92E-4 &   2.62E-2 &   6.66E-4 &	   2.35 &        1.25 &        2.28      \\\hline
10 &		  63 &	   27.96 &		   4.75E-4 &   2.16E-2 &   4.68E-4 &	   1.69 &       0.88 &        1.59  \\\hline
11 &	 	  105 &   189.66 &      1.65E-4 &   1.28E-2 &   1.58E-4 &	   2.04 &        1.00 &        2.09
 \\\hline\hline
\end{tabular}}
\vskip 0.5cm
\resizebox{0.47\textheight}{!}{
\begin{tabular}{|c|c|c|c|c|c|c|c|c|}
\cline{2-9}
\multicolumn{1}{c|}{}  & \multicolumn{8}{c|}{\textbf{\textbf{Anisotropic Permeability: }\textit{VAG Lump}}} 
\\\hline
\multicolumn{9}{|l|}{\textit{Hexahedral Meshes}}
\\ \hline
\multicolumn{1}{|l|}{\textbf{Key}}
& \multicolumn{1}{|l|}{\textbf{Iter}} & \multicolumn{1}{|l|}{\textbf{CPU}} & \multicolumn{1}{|l|}{\textbf{$err_{sol}$}}  
& \multicolumn{1}{|l|}{\textbf{$err_{grad}$}} & \multicolumn{1}{|l|}{\textbf{$err_{jump}$}}
& \multicolumn{1}{|l|}{\textbf{$\alpha_{sol}$}} & \multicolumn{1}{|l|}{\textbf{$\alpha_{grad}$}} & \multicolumn{1}{|l|}{\textbf{$\alpha_{jump}$}}
\\ \hline
1 &	  	  7 &	   3.90E-3 &	      9.09E-3 &	   2.01E-2 &	   9.06E-3 &	   1.89 &        1.99 &        1.89		\\\hline
2 &	 	  9 &	   5.15E-2 &		  2.46E-3 &	   5.06E-3 &	   2.44E-3 &	   1.89 &        1.99 &        1.89 \\\hline
3 &	     15 &  0.66 &		     6.37E-4 &	   1.27E-3 &	   6.34E-4 &	   1.95 &        2.00 &        1.95   \\\hline
4 &    	  26 &   6.39 &	          1.62E-4 &	   3.17E-4 &	   1.61E-4 &	   1.97 &        2.00 &        1.97 \\\hline
5 &		  47 &   62.19 &        4.09E-5 &	   7.93E-5 &	   4.07E-5 &	   1.99 &        2.00 &        1.99
 \\\hline\hline
 \multicolumn{9}{|l|}{\textit{Tetrahedral Meshes}}
\\ \hline
6 &		  7 &	   2.11E-3 &		  2.75E-2 &	  0.13 &        2.73E-2 &	   1.95 &       0.99 &        1.94	  	\\\hline
7 &		  8 &   2.00E-2 &		  7.14E-3 &	   6.76E-2 &   7.10E-3 &	   1.95 &       0.99 &        1.94		 \\\hline
8 &  		  15 &  0.38 &			   2.60E-3 &   3.92E-2 &   2.58E-3 &        1.35 &       0.73 &        1.35   \\\hline
9 &       21 &   1.02 &			   1.36E-3 &   2.84E-2 &   1.35E-3 &	   2.48 &        1.24 &        2.49     \\\hline
10 &		 25 &   2.24 &	           9.40E-4 &   2.22E-2 &   9.33E-4 &	   1.66 &        1.10 &        1.67    \\\hline
11 &	 	  41 &   13.91 &         3.15E-4 &   1.36E-2 &   3.13E-4 &	  2.11 &	       0.95 &        2.11
 \\\hline\hline
\end{tabular}}
\end{center}
\caption{ Anisotropic test case. \textbf{Key} refers to the mesh defined in table \ref{table1}; \textbf{Iter} is the number of solver iterations; 
\textbf{CPU} refers to the solver CPU time in seconds; \textbf{$err_{sol}, err_{grad}, err_{jump}$} are the respective $L^2$-errors as defined above;
\textbf{$\alpha_{sol}, \alpha_{grad}, \alpha_{jump}$} are the orders of convergence w.r.t. $\#\cells^{-{1\over 3}}$ 
of the solution, of the gradient and of the jump, respectively.}
\label{table3}
\end{table}

The test case shows that, on cartesian grids, we obtain, as classically expected, convergence of order 2 for both, 
the solution and it's gradient. For tetrahedral grids, we obtain convergence of order 2 for the solution 
and convergence of order 1 for it's gradient. We observe that the VAG scheme is more efficient 
then the HFV scheme and this observation gets more obvious with increasing anisotropy. 
Comparing the precision of the discrete solution (and it's gradient) for VAG and HFV on a given mesh, 
we see that on hexahedral meshes, the advantage is on the side of VAG, whereas on tetrahedral meshes HFV 
is more precise (but much more expensive). On a given mesh, HFV is usually (see \cite{GSDFN}) 
more accurate than VAG both for tetrahedral and hexahedral meshes. This is not the case for our test cases on Cartesian meshes 
maybe due to the higher number for VAG 
than for HFV of d.o.f. at the interfaces $\Gamma_\alpha$ on the matrix side. 
It is also important to notice that there is literally no difference between VAG 
with finite element respectively lumped \emph{mf}-fluxes concerning accuracy and convergence rate.

\section{Conclusion}

In this work, we extended the framework of gradient schemes (see \cite{Eymard.Herbin.ea:2010}) to the model problem \eqref{modeleCont} of 
stationary Darcy flow through fractured porous media and gave numerical analysis results for this general framework.

The model problem (an extension to a network of fractures 
of a PDE model presented in \linebreak
\cite{FNFM03}, \cite{MJE05} and \cite{ABH09}) 
takes heterogeneities and anisotropy 
of the porous medium into account and involves a complex network of planar fractures, 
which might act either as barriers or as drains.

We also extended the VAG and HFV schemes to our model, 
where fractures acting as barriers force us to allow for pressure jumps across the fracture network. 
We developed two versions of VAG schemes, the conforming finite element version and the non-conforming control volume version, the latter particularly adapted for the treatment of material interfaces (cf. \cite{EHGM2012}). 
We showed, furthermore, that both versions of VAG schemes, 
as well as the proposed non-conforming HFV schemes, are incorporated by the gradient scheme's framework.
Then, we applied the results for gradient schemes on VAG and HFV to obtain convergence, and, in particular, convergence of order 1 
for "piecewise regular" solutions.

For implementation purposes and in view of the application to multi-phase flow, we also proposed a uniform Finite Volume formulation for VAG and HFV schemes.
The numerical experiments on a family of analytical solutions show that 
the VAG scheme offers a better compromise between accuracy and CPU time than the HFV scheme 
especially for anisotropic problems. 

\vskip 1cm 
\noindent{\bf Acknowledgements}: the authors would like to thank TOTAL for its financial support 
and for allowing the publication of this work.

\end{document}